\documentclass[12pt]{amsart}

\usepackage{amsaddr}
\usepackage{color}
\usepackage{comment}
\usepackage{lineno}
\usepackage{ulem}

\def\B{\mathbb{B}}
\def\R{\mathbb{R}}
\def\RN{\mathbb{R}^N}

\def\calL{\mathcal{L}}

\def\e{\varepsilon}
\def\ul{{ ul}}
\def\loc{{loc}}

\def\bu{\bar{u}}

\def\bw{\bar{w}}

\def\XXint#1#2#3{{\setbox0=\hbox{$#1{#2#3}{\int}$}
\vcenter{\hbox{$#2#3$}}\kern-.5\wd0}}

\newcommand{\di}{\mathop{}\!\mathrm{d}}

\DeclareMathOperator{\esssup}{ess\,sup}

\theoremstyle{plain}
	\newtheorem{theorem}{Theorem}[section]
	\newtheorem{lemma}[theorem]{Lemma}
	
	\newtheorem{definition}[theorem]{Definition}
	\newtheorem{proposition}[theorem]{Proposition}
	\newtheorem{remark}[theorem]{Remark}
	
	\newtheorem{problem}[theorem]{Problem}
	
\theoremstyle{plain}

\makeatletter
 \@addtoreset{equation}{section}
\makeatother

\begin{document}
\title[Nonuniqueness]{
Non-uniqueness of positive solutions for\\
supercritical semilinear heat equations\\ without scale invariance
}

\author{Kotaro Hisa}
\address[Kotaro Hisa]{
Department of Applied Mathematics, Faculty of Sciences, Fukuoka University,\\ 8-19-1, Nanakuma, Jonan-ku, Fukuoka 814-0180, Japan,\\
{\rm{\texttt{hisak@fukuoka-u.ac.jp}}}
}

\author{Yasuhito Miyamoto}
\thanks{ORCiD of YM is 0000-0002-7766-1849}
\thanks{YM was supported by JSPS KAKENHI Grant Number 24K00530.}
\address[Yasuhito Miyamoto]{Graduate School of Mathematical Sciences, The University of Tokyo,\\
 3-8-1 Komaba, Meguro-ku, Tokyo 153-8914, Japan,\\
{\rm{\texttt{miyamoto@ms.u-tokyo.ac.jp}}}
}

\begin{abstract}
We establish nonuniqueness of solutions for Cauchy problems of semilinear heat equations with a wide class of nonlinearities.
Specifically, we consider
\[
\begin{cases}
\partial_tu-\Delta u=f(u), & x\in\RN,\ t>0,\\
u(x,0)=u_0(x), & x\in\RN,
\end{cases}
\]
where $N>2$.
We assume that the growth rate of $f$ is less than the Joseph-Lundgren exponent for $N>10$ and it satisfies certain assumptions guaranteeing {the existence of} a positive radial singular stationary solution $u^*$.
We prove that if $u_0=u^*$, then the problem has at least two positive solutions, namely $u^*$ and $u(t)$ which satisfies $u(t)\in L_{loc}^{\infty}(0,t_0;L^{\infty}(\RN))$ for some $t_0>0$ and
$$
u(t)\to u^*\quad\text{in}\ L^{\gamma}_{ul}(\RN)\quad\text{as}\ t\to 0^+
$$
for $1\le \gamma<N(p_f-1)/2$, where $p_f:=\lim_{u\to\infty}uf'(u)/f(u)$ is a growth rate of $f$.
Hence, nonuniqueness problem can be reduced to the existence problem of a positive radial singular stationary solution.
The method of construction of $u(t)$ is based on the monotonicity argument.
Transformations of forward self-similar solutions for $f(u)=u^p$ and $e^u$ play a crucial role.
\end{abstract}

\date{\today}

\maketitle

\section{Introduction and main results}
We are concerned with Cauchy problem for the semilinear heat equation
\begin{equation}\label{S1E1}
	\left\{
	\begin{aligned}
		&\partial_t u  -\Delta  u =  f(u),	\quad && x\in \RN,\,\, t>0,\\
		&u(x,0) =  u_0(x),		\quad && x\in \RN,
	\end{aligned}
	\right.
\end{equation}
where  $N>2$ and $u_0$ is a nonnegative measurable function on $\mathbb{R}^N$.
The aim of this paper is to establish nonuniqueness of positive solutions of problem~\eqref{S1E1} for a wide class of nonlinearities if problem~\eqref{S1E1} has a positive radial singular stationary solution.

Let us explain notations.
Let $p_S$ denote the critical Sobolev exponent, {\it i.e.},
\[
p_S :=
\begin{cases}
\displaystyle{\frac{N+2}{N-2}} & \text{if}\ N>2,\\
\infty & \text{if}\ N=1,2.
\end{cases}
\]	
Throughout the present paper, we make the following assumptions on $f$:
\begin{itemize}
\item[(A1)]  $f\in C^1[0,\infty)\cap C^2(0,\infty)$, $f(0)=0$, and $f'(0)=0$;
\item[(A2)]  $f'(u)>0$ for $u>0$ and $f''(u)>0$ for $u>0$;
\item[(A3)] There exists a limit
$$
q_f:=\lim_{u\to\infty}\frac{f'(u)^2}{f(u)f''(u)};
$$
\item[(A4)] The following holds:
\begin{equation*}
Q(u):=uf(u) - (p_S+1) \int_{0}^u f(s) \, \di s \ge 0\quad\mbox{for}\ u\ge 0.
\end{equation*}
\end{itemize}
We define $f(u)=0$ for $u<0$.
In \cite{FI18} $A$ is defined by $A=\lim_{u\to\infty}f'(u)F(u)$, where
\begin{equation}\label{F}
F(u)=\int_u^{\infty}\frac{ds}{f(s)} \quad\mbox{for}\quad u>0.
\end{equation}
Then, when (A3) holds, by L'Hospital rule we have
$$
A=\lim_{u\to\infty}\frac{F(u)}{1/f'(u)}=\lim_{u\to\infty}\frac{f'(u)^2}{f(u)f''(u)}=q_f.
$$
It is proved in \cite[Remark 1.1]{FI18} that if the limit $A$ exists, then $A\ge 1$.
Hence, $q_f\ge 1$ also holds if (A3) holds.
The growth rate of $f$ can be defined by
$$p_f:=\lim_{u\to\infty}\frac{uf'(u)}{f(u)}.
$$
Then, by L'Hospital rule we have
$$
\frac{1}{p_f}=\lim_{u\to\infty}\frac{f(u)/f'(u)}{u}=\lim_{u\to\infty}\left(1-\frac{f(u)f''(u)}{f'(u)^2}\right)=1-\frac{1}{q_f},
$$
and hence $1/p_f+1/q_f=1$.
The exponent $q_f$ is the conjugate exponent of the growth rate $p_f$.
Using the $q_f$-exponent, we can treat power nonlinearities ($1<q_f<\infty$) and super-power nonlinearities ($q_f=1$), including exponential nonlinearities, in a unified way.
We later use the conjugate exponent of $p_S$, {\it i.e.,}
$$
q_S:=\frac{p_S}{p_S-1}=\frac{N+2}{4}.
$$
The exponent $q_{S}$ can be formally defined for $N\ge 1$.
However, $q_{S}$ becomes effective only in the case $N>2$, since $q_{S}>1$ for $N>2$ and $q_f\ge 1$ always holds.

Typical examples are $f(u)=u^{\beta}$, $\beta\ge p_S$, and $f(u)=u^{\beta}e^{u^{\gamma}}$, $\beta\ge p_S$ and $\gamma>1$.
Other examples are given in Section~2.

We need the so-called Joseph-Lundgren exponent
$$
p_{JL}:=
\begin{cases}
\displaystyle{1+\frac{4}{N-4-2\sqrt{N-1}}} & \text{if}\ N>10,\\
\infty & \text{if}\ 2\le N\le 10.
\end{cases}
$$
The conjugate exponent of $p_{JL}$ is given by
$$
q_{JL}:=\frac{p_{JL}}{p_{JL}-1}=\frac{N-2\sqrt{N-1}}{4}.
$$
The exponent $q_{JL}$ can be formally defined for $N\ge 1$.
However, $q_{JL}$ becomes effective only in the case $N>10$, since $q_{JL}>1$ for $N>10$ and $q_f\ge 1$ always holds.

For $x\in\mathbb{R}^N$ and $\rho>0$, set $B(x,\rho):=\left\{ y\in\RN;\ |x-y|<\rho\right\}$.
We define a uniformly local $L^p$ space as follows:
For $1\le p\le\infty$, 
\[
L^p_{\ul}(\RN):=\left\{ u\in L^p_{\loc}(\RN);\ \left\|u\right\|_{L^p_{ul}(\RN)}<\infty\right\},
\]
where
$$
\left\|u\right\|_{L^p_{\ul}(\RN)}:=
\begin{cases}
\displaystyle{\esssup\displaylimits_{z\in\RN}\left\|u\right\|_{L^{\infty}(B(z,1))}} & \textrm{if}\ p=\infty,\\
\displaystyle{\sup_{z\in\RN}\left(\int_{B(z,1)}|u(x)|^p \, \di x\right)^\frac{1}{p}} & \textrm{if}\ 1\le p<\infty.
\end{cases}
$$
It is obvious that $L^{\infty}_{\ul}(\RN)=L^{\infty}(\RN)$ and that $L^{p_1}_{\ul}(\RN)\subset L^{p_2}_{\ul}(\RN)$ if $1\le p_2\le p_1 <\infty$.
Let $\calL^p_{\ul}(\RN)$ denote the closure of the space of bounded uniformly continuous functions $BUC(\RN)$ in the space $L^p_{\ul}(\RN)$, {\it i.e.},
\begin{equation}\label{BUC}
\calL^p_{\ul}(\RN):=\overline{BUC(\RN)}^{\left\|\,\cdot\,\right\|_{L^p_{\ul}(\RN)}}.
\end{equation}

Let $G_t= G_t(x)$ be the fundamental solution of 
\[
\partial_t v -\Delta v = 0 \quad \mbox{in} \quad \mathbb{R}^N\times (0,\infty),
\]
that is, 
\[
G_t(x) := \frac{1}{(4\pi t)^\frac{N}{2}} \exp \left(-\frac{|x|^2}{4t}\right)
\]
for $x\in \mathbb{R}^N$ and $t>0$.
For $u_0\in L^1_{\ul} (\mathbb{R}^N)$, let $S(t)u_0$ be
\[
[S(t)u_0](x) := \int_{\mathbb{R}^N} G_t(x-y) u_0(y) \, \di y
\]
for $x\in \mathbb{R}^N$ and $t>0$.

We are in a position to define a solution of problem~\eqref{S1E1}.
\begin{definition}\label{S1D1}
Let $u$ be a nonnegative measurable function on $\RN\times[0,T)$, where $T\in(0,\infty]$.
\begin{itemize}
\item[(i)] We say that $u$ is a solution of problem~\eqref{S1E1} in $\mathbb{R}^N\times [0,T)$ if
\begin{equation}\label{S1D1E2}
\begin{split}
\infty>u(x,t)
&=\int_{\mathbb{R}^N}G_t(x-y )u_0(y) \, \di y+\int_0^t\int_{\mathbb{R}^N}G_{t-s}(x-y)f(u(y,s)) \, \di y\di s
\end{split}
\end{equation}
for a.a.~$(x,t)\in\R^N\times [0,T)$.
\item[(ii)] If  $u$ satisfies \eqref{S1D1E2} with $=$ replaced by $\ge$, then we say that $u$ is a supersolution in $\mathbb{R}^N\times [0,T)$.
\end{itemize}
\end{definition}

Under the assumptions {\rm (A1)}--{\rm (A4)} with $q_f<q_S$, the following positive radial singular stationary solution exists:
\begin{proposition}\label{S1P1}
Let $N>2$.
Suppose {\rm (A1)}--{\rm (A4)}.
Suppose, in addition, that $q_f<q_S$.
The problem
\begin{equation}\label{SS}
\begin{cases}
-\Delta u=f(u), & x\in\R^N\setminus\{0\},\\
u(x)\ge 0, & x\in\R^N\setminus\{0\},\\
u(x)\ \mbox{is radial},\\
u(x)\to\infty\quad\mbox{as}\quad |x|\to 0^+,
\end{cases}
\end{equation}
has a unique solution $u^*\in C^2(\R^N\setminus\{0\})$, which is called a singular solution.
Moreover, $u^*\in \calL^\gamma_{\ul}(\R^N)$ for $1\le \gamma<\gamma^*$,
$u^{*\prime}(r)<0$ for all $r>0$ and $u^*$ satisfies
\[
u^*(x)=F^{-1}\left[\frac{|x|^2}{2N-4q_f}(1+o(1))\right]\quad\mbox{as}\quad |x|\to 0^+,
\]
where 
\[
1 \le \gamma^* := 
\begin{cases}
\displaystyle{\frac{N}{2}\frac{1}{q_f-1}} & \mbox{if} \quad q_f>1,\\
\infty &  \mbox{if} \quad q_f=1,\\
\end{cases}
\]
and
$F^{-1}$ is the inverse function of $F$ defined by \eqref{F}.
\end{proposition}
We will prove the following lemma in Section~3.
\begin{lemma}\label{S10L1}
Let $u^*$ be as in Proposition~{\rm \ref{S1P1}}.
Then, $u^*$ is a positive singular stationary solution of problem~\eqref{S1E1} in the sense of Definition~{\rm \ref{S1D1}}.
In particular, 
\[
\infty>u^*(x)=[S(t)u^*](x)+\int_0^t[S(t-s)f(u^*)](x)\, \di s
\]
for all $x\in\RN\setminus\{0\}$ and $t\in(0,\infty)$.
\end{lemma}

The main result of the present paper is the following:
\begin{theorem}[Nonuniqueness (supercritical)]\label{S1T1}
Let $N>2$.
Suppose {\rm (A1)}--{\rm (A4)}.
Let $u^*$ be the singular solution of problem~\eqref{SS} given in Proposition~{\rm \ref{S1P1}}.
Suppose, in addition, the following {\rm (A5)} and {\rm (A6)}:
\begin{itemize}
\item[{\rm (A5)}] $q_{JL}<q_f<q_S\quad$  (\lq\lq almost" equivalently, $p_S<p_f<p_{JL}$);
\end{itemize}
\begin{itemize}
\item[{\rm (A6)}] If $q_f=1$, then $\displaystyle\frac{f'(u)^2}{f(u)f''(u)}\le 1$ for large $u>0$.
\end{itemize}
Then, problem~\eqref{S1E1} with $u_0=u^*$ has at least two positive solutions in the sense of Definition~{\rm \ref{S1D1}}, and hence uniqueness does not hold.
Specifically, one solution is $u(x,t)=u^*$ a.e.~in $\RN\times [0,\infty)$.
The other solution satisfies $u(x,t)\in L^{\infty}_{loc}((0,t_0),L^{\infty}(\RN))$ for some $t_0>0$, $u(x,t)$ satisfies problem~\eqref{S1E1} in the classical sense for $(x,t)\in\RN\times (0,t_0)$ and
\begin{equation}\label{S1T1S1}
\lim_{t\to 0^+}\left\|u(t)-u^*\right\|_{L^{\gamma}_{ul}(\RN)}=0,
\end{equation}
provided that  $1\le \gamma < \gamma^*$, where
$\gamma^*$ is defined in Proposition~{\rm \ref{S1P1}}.
\end{theorem}
Note that if $q_f=1$, then (A5) is equivalent to $2<N<10$.\\

In the pure power case $f(u)=u^{\beta}$ the conclusion of Theorem~\ref{S1T1} holds for $p_S<\beta<p_{JL}$.
The model case in the present paper is $f(u)=u^{\beta}e^{u^{\gamma}}$, $\beta\ge p_S$ and $\gamma>1$.
In this case we see in Section~2 that the conclusion of Theorem~\ref{S1T1} holds for $2<N<10$.

The second main theorem is nonuniqueness for critical and subcritical cases.
\begin{theorem}[Nonuniqueness (critical/subcritical)]\label{S1T2}
Let $N>2$, and let
$$
p_0:=\frac{N}{N-2}\quad\text{and}\quad
q_0:=\frac{p_0}{p_0-1}=\frac{N}{2}.
$$
Suppose {\rm (A1)}--{\rm (A3)} and the following {\rm (A7)} and {\rm (A8)}:
\begin{itemize}
\item[{\rm (A7)}] $q_S\le q_f<q_0$\quad (equivalently, $p_0<p_f\le p_S$),
\item[{\rm (A8)}] Problem~\eqref{S1E1} has a positive radial singular stationary solution $u^*$ such that
\begin{multline}\label{A8E1}
u^*(x)=F^{-1}\left[\frac{|x|^2}{2N-4q_f}(1+\theta(|x|))\right]\quad\text{as}\ |x|\to 0^+,\\
\text{where $\theta(r)$ satisfies $\theta\in C^1(0,r_1)$ for some $r_1>0$, $\lim_{r\to 0^+}\theta(r)=0$ and $\lim_{r\to 0^+}r\theta'(r)=0$.}
\end{multline}

\end{itemize}
Then, problem~\eqref{S1E1} with $u_0=u^*$ has at least two positive solutions in the sense of Definition~{\rm\ref{S1D1}}, and hence uniqueness does not hold.
Specifically, one solution is $u(x,t)=u^*$ a.e.~in $\RN\times[0,\infty)$.
The other solution satisfies $u(x,t)\in L^{\infty}_{loc}((0,t_0),L^{\infty}(\RN))$ for some $t_0>0$, $u(x,t)$ satisfies problem~\eqref{S1E1} in the classical sense for $(x,t)\in\RN\times (0,t_0)$ and
\eqref{S1T1S1} holds provided that $1\le \gamma<\gamma^*$.
\end{theorem}

Because of (A8), nonuniqueness problem can be reduced to the following:
\begin{problem}\label{S1PR1}
Let $N>2$.
Assume that {\rm (A1)}--{\rm (A3)} and {\rm (A7)} hold.
Find conditions on $f$ such that problem~\eqref{S1E1} has a positive radial singular stationary solution $u^*$ satisfying \eqref{A8E1}.
\end{problem}
Singular solutions satisfying \eqref{A8E1} are recently constructed near the origin for $p_0<p_f<p_S$ in \cite{FI25}.
The authors of \cite{MN23b} showed that every positive radial solution diverging to $\infty$ as $|x|\to 0$ satisfies \eqref{A8E1} in the subcritical case $p_0<p_f<p_S$.
This will be published elsewhere.
On the other hand, it is well-known that $\Delta u+u^p=0$ has a positive radial singular solution that does not satisfy \eqref{A8E1} in the critical case $p=p_S$ ($q=q_S$).
Note that in the supercritical case the singular solution is unique and it always satisfies \eqref{A8E1} (Proposition~\ref{S2P1}).
Therefore, the case $p=p_S$ is exceptional in the range $p>p_0$ ($q<q_0$).

Even if a singular solution $u^*$ is constructed near the origin, it is not easy to determine whether $u^*(r)$ has a zero or not in the case $p_0<p_f< p_S$.
Pohozaev's identity does not seem to work.
Problem~\ref{S1PR1} is an open question at this moment.
\bigskip

Let us recall previous research about nonuniqueness.
Let $N>2$.
The pure power case
\begin{equation}\label{S1E19}
\begin{cases}
\partial_tu-\Delta u=|u|^{p-1}u, & x\in\RN,\ t>0,\\
u(x,0)=u_0(x), & x\in\RN
\end{cases}
\end{equation}
has been intensively studied by many researchers.
The equation \eqref{S1E19} is invariant under the transformation
$$
u(x,t)\mapsto u_{\lambda}(x,t)=\lambda^{\frac{2}{p-1}}u(\lambda x,\lambda^2 t)\quad\text{for all}\ \lambda>0.
$$
We call a solution $u$ of problem~\eqref{S1E19} the forward self-similar solution if $u(x,t)=u_{\lambda}(x,t)$ for all $\lambda>0$.
Let $\gamma_c$ be defined by 
$$
\gamma_c:=\frac{N(p-1)}{2}.
$$
Note that $\gamma^*={N/2(q_f-1)}={N(p_f-1)/2}=\gamma_c$ if $p=p_f<\infty$.
Haraux-Weissler~\cite{HW82} constructed a positive forward self-similar solution $u(t)$ of problem~\eqref{S1E19} that is in $C([0,T],L^{\gamma}(\RN))$ for $1\le \gamma<\gamma_c$ by finding a positive radial solution of
\begin{equation}\label{S1E18}
\Delta\varphi+\frac{1}{2}\xi\cdot\nabla\varphi+\frac{\varphi}{p-1}+|\varphi|^{p-1}\varphi=0.
\end{equation}
Let
$$
p_F:=1+\frac{2}{N}.
$$
If $p_F<p<p_S$, then $\left\|u(t)\right\|_{L^{\gamma}(\RN)}\to 0$ as $t\to 0^+$, and hence problem~\eqref{S1E19} with $u_0=0$ has two solutions, namely $u(t)$ and the trivial one.
Therefore, uniqueness does not hold in $C([0,T],L^{\gamma}(\RN))$.
Since self-similar solutions are closely related to nonuniqueness of solutions and asymptotic behaviors of solutions to problem~\eqref{S1E19}, the elliptic equation~\eqref{S1E18} has attracted much attention.
The details of the structure of radial solutions of equation~\eqref{S1E18} can be found in \cite{N06}.
Nonradial nonnegative solutions of equation~\eqref{S1E18} were constructed for $p_F<p<p_S$ in Naito~\cite{N04}.
If $1<p\le p_F$, then it is well-known that problem~\eqref{S1E19} does not have global-in-time positive solutions (see \cite{F66}).
Therefore, problem~\eqref{S1E18} has no positive solution for $1<p\le p_F$, since a self-similar solution gives a global-in-time solution.
For sign-changing radial solutions of equation~\eqref{S1E18} in the case $p_F<p<p_S$, see Cazenave {\it et.al.}~\cite{CDNW20}.

Let $B\subset\RN$, $N>2$, denote the unit ball.
Ni-Sacks~\cite{NS85} studied Cauchy problem
\begin{equation}\label{S1E20}
\begin{cases}
\partial_tu-\Delta u=u^{p_0}, & x\in B,\ t>0,\\
u=0, & x\in\partial B,\ t>0,\\
u(x,0)=u_0, & x\in B.
\end{cases}
\end{equation}
They constructed a positive singular stationary solution $u^*$ that belongs to $L^{p_0}(B)$.
Moreover, they showed that problem~\eqref{S1E20} with $u_0=u^*$ has a positive bounded solution $u(t)\in C([0,T],L^{p_0}(B))\cap L^{\infty}_{loc}((0,T),L^{\infty}(\RN))$.
Therefore, uniqueness of positive solutions in $C([0,T],L^{p_0}(B))$ does not hold.
When the domain is $\RN$, nonuniqueness of positive solutions of problem~\eqref{S1E19} in $C([0,T],L^{p_0}(\RN))$ was proved in Terraneo~\cite{T02}.
Later, infinitely many solutions of problem~\eqref{S1E19} in $C([0,T],L^{p_0}(\RN))$ with one initial data were constructed in Matos-Terraneo~\cite{MT03} and Takahashi~\cite{T21}.

Let $u^*$ denote the explicit positive radial singular stationary solution of problem~\eqref{S1E19}, {\it i.e.,}
\begin{equation}\label{L}
u^*(x)=L|x|^{-\frac{2}{p-1}},\quad L:=\left\{\frac{2}{p-1}\left(N-2-\frac{2}{p-1}\right)\right\}^{\frac{1}{p-1}}
\end{equation}
for $p>p_0$.
In Galaktionov-Vazquez \cite{GV97} a positive bounded solution of problem~\eqref{S1E19} with $u_0=u^*$ in the sense of the proper solution was constructed for $p_S<p<p_{JL}$, and hence uniqueness does not hold.
The definition and properties of the proper solution can be found in \cite[Section~2]{GV97}.
They also studied Cauchy-Dirichlet problem
\begin{equation}\label{S1E21}
\begin{cases}
\partial_tu-\Delta u=2(N-2)e^u, & x\in B,\ t>0,\\
u=0, & x\in\partial B,\ t>0,\\
u(x,0)=u_0, & x\in B.
\end{cases}
\end{equation}
This problem has the explicit positive singular stationary solution $u^*(|x|)=-2\log |x|$.
If $2<N<10$, then a bounded solution of problem~\eqref{S1E21} with $u_0=u^*$ in the sense of the proper solution was constructed in \cite{GV97}.
Hence, uniqueness does not hold for $2<N<10$.
See also \cite{PV95} for problem~\eqref{S1E21}.

When the domain is a ball in $\R^2$, Ioku {\it et.al.}~\cite{IRT19} and Ibrahim {\it et.al.}~\cite{IKNW21} proved nonuniqueness of positive solutions of Cauchy-Dirichlet problem.
Specifically, they constructed a positive singular stationary solution $u^*$ and a bounded solution with initial data $u^*$ when the nonlinearity $f$ has an exponential growth. 

Nonuniqueness was studied for the equation $\partial_tu=\Delta u+|x|^{\beta}|u|^{p-1}u$ in \cite{T20,CITT24,GW24}.
For nonuniqueness for other equations, see \cite[Remark 3]{IKNW21} and references therein.

Theorem~\ref{S1T1} unifies two nonuniqueness results of \cite{GV97}, since $p_S<p_f<p_{JL}$ in the power case ($p_f\neq 1$) and $2<N<10$ in the super-power case $(p_f=1)$.
However, we will construct the solution defined in Definition~\ref{S1D1} and do not directly use the proper solution.

\bigskip
Let us recall previous research about uniqueness.
A natural question is to find the largest function spaces such that uniqueness holds.
Uniqueness of solutions of problem~\eqref{S1E19} have been studied by many researchers.
Brezis-Cazenave~\cite{BC96} showed that the solution of problem~\eqref{S1E19}, which was constructed in Weissler~\cite{W80}, is unique in $C([0,T],L^{\gamma}(\RN))$ if either $\gamma=\gamma_c>p$ or both $\gamma>\gamma_c$ and $\gamma\ge p$.
When $\gamma$ is in these ranges, we see that $u^*(x)\chi_B(x)\not\in L^{\gamma}(\RN)$ and hence the singularity of $u^*$ is too strong to obtain uniqueness.
As mentioned above, uniqueness in $C([0,T],L^{p_0}(\RN))$ does not hold for $p=p_0$.
We remark that $p=p_0\iff \gamma_c=p$, and hence this case is a borderline case of \cite{BC96}.

We briefly introduce the Lorentz space $L^{p,\infty}(\RN)$ to mention the uniqueness result of Chikami~{\it et.al.} \cite{CITT24}.
For $0<p\le\infty$,
$$
L^{p,\infty}(\RN):=\left\{ u\in M(\RN);\ \left\|u\right\|_{L^{p,\infty}(\RN)}:=\sup_{\lambda>0}\lambda\left|\left\{x\in\RN;\ |u(x)|>\lambda\right\}\right|^{\frac{1}{p}}<\infty\right\},
$$
where $M(\RN)$ denote the set of measurable functions on $\RN$ and $|S|$ denotes the Lebesgue measure of the set $S\subset\RN$.
The space $\calL^{p,\infty}(\RN)$ is the closure of $L^{p,\infty}(\RN)\cap L^{\infty}(\RN)$ in the space $L^{p,\infty}(\RN)$.
The uniqueness of the solution of \eqref{S1E19} in $C([0,T],\calL^{\gamma_c,\infty}(\RN))$ was proved for $p>p_S$ in \cite[Theorem 1.3 i]{CITT24}.
We can check that $u^*\in L^{\gamma_c,\infty}(\RN)\setminus\calL^{\gamma_c,\infty}(\RN)$, and hence their theorem cannot apply to the case $u_0=u^*$.

In the case of the general nonlinearity \eqref{S1E1} there are few theorems about uniqueness, although existence and nonexistence theorems were recently developed in \cite{FI18,FHIL23,FHIL24}.

The case $N=2$ is out of our scope.
However, in Ioku~{\it et.al.}~\cite{IRT15} the uniqueness of the solution of problem~\eqref{S1E1} in $C([0,T],{\rm exp}L_0^2(\R^2))$ was established if $f$ satisfies
$$
|f(x)-f(y)|\le C|x-y|\left(e^{\lambda x^2}+e^{\lambda y^2}\right).
$$
Here, ${\rm exp}L_0^2(\R^2)$ is the closure of $C_0^{\infty}(\R^2)$ in the Orlicz space ${\rm exp}L^2(\R^2)$.
The uniqueness of the positive radial singular solution on $B$ is an open problem.
It was proved in \cite{IRT15,IRT19} that for a certain positive singular stationary solution $u^*$ on $B$ of a certain explicit nonlinearity $f$, $u^*\in{\rm exp}L^2(B)\setminus{\rm exp}L^2_0(B)$.

\begin{remark}
As mentioned above, Fujishima {\it et.al.}~\cite{FHIL24} obtained rather sharp existence and nonexistence criteria.
It follows from \cite[Corollary 1.1]{FHIL24} that there exists a unique $\lambda_0>0$ such that problem~\eqref{S1E19} with $u_0(x)=F^{-1}[{\lambda^{-1}|x|^2}]\chi_{B}(x)$ has a solution for $0<\lambda<\lambda_0$ and no solution for $\lambda>\lambda_0$.
However, the exact value of $\lambda_0$ is an open problem.
Theorem~{\rm \ref{S1T1}} indicates that $\lambda_0\ge 2N-4q_f$ if $q_{JL}<q<q_S$.
When $q_f=1$, the authors of~\cite{HM25} proved that $\lambda_0=2N-4$ under additional assumptions on solutions and the nonlinearity.
When $f(u)=u^p$, $p\ge p_{JL}$, it was proved in \cite{GV97} that $\lambda_0=2N-4q$ in the framework of the proper solution, where $q:=p/(p-1)$.
If $p_S<p<p_{JL}$, then $\lambda_0>2N-4q$ by Souplet-Weissler~\cite{SW03}.
\end{remark}

Let us mention technical details.
We will see in Section~3 that the positive singular solution $u^*$ given in Proposition~\ref{S1P1} is a positive singular stationary solution of problem~\eqref{S1E1} in the sense of Definition~\ref{S1D1}.
We also construct a positive bounded solution $u(t)$ of problem~\eqref{S1E1} with $u_0=u^*$.
This solution is in $L^{\infty}_{loc}((0,T),L^{\infty}(\RN))$, and hence $u(t)$ is different from $u^*$.

The method of the construction of $u(t)$ is based on the comparison principle and not on the contraction mapping theorem.
We construct a supersolution $v(t)$ defined by \eqref{S4E7} with \eqref{S4E4+}, combining the singular solution and a transformation of a self-similar solution.
Let $u_{0n}:=\min\{u^*,n\}$.
Then, $u_{0n}\in BUC(\RN)$, and hence \eqref{S1E1} with $u_0=u_{0n}$ has a classical solution $u_n(t)$ for $0\le t<t_0$.
Since $u_n$ is a classical solution, the maximum principle works and we obtain the monotone sequence
$$
u_1(t)<u_2(t)<\cdots <u_n(t)<\cdots <v(t)
$$
for $0<t<t_0$.
Here, the inequality $u_n(t)<v(t)$ follows from the definition of $u_n$.
Letting $n\to\infty$, we obtain the limit function
$$
u(t):=\lim_{n\to\infty}u_n(t),
$$
which satisfies \eqref{S1D1E2} by the monotone convergence theorem.
Since $u(t)\le v(t)$ and $v(t)$ is bounded for each $t>0$, $u(t)$ is the desired bounded solution.
This method is inspired by the definition of the proper solution which is explained in \cite{GV97}.
However, we have to construct a bounded solution $u(t)$ in the sense of Definition~\ref{S1D1}.\\

This paper consists of six sections.
In Section~2 we give three examples of $f$.
In Section~3 we construct a positive radial singular solution $u^*$ and prove Proposition~\ref{S1P1} and Lemma~\ref{S10L1}.
In Section~4 we recall basic properties of forward self-similar solutions of problem~\eqref{S1E1} in the cases $f(u)=u^p$ and $e^u$.
Combining the self-similar solution and the singular stationary solution, we construct a supersolution $v$ of problem~\eqref{S1E1} in the sense of Definition~\ref{S1D1}.
In Section~5 we construct a bounded solution of problem~\eqref{S1E1} with $u_0=u^*$, using the supersolution $v$.
Then we prove Theorem~\ref{S1T1}.
In Section~6 we briefly prove Theorem~\ref{S1T2}, modifying the proof of Theorem~\ref{S1T1}.

\section{Examples}
\subsection{Example 1}
The first example is
$$
f(u)=u^{\beta}e^{u^{\gamma}},\quad \beta\ge p_S,\quad \gamma>1.
$$
We easily see that {\rm (A1)}, {\rm (A2)} and {\rm (A4)} hold.
See \cite[Example 1]{HM25} for details.
By direct calculation we have
\begin{equation}\label{S2E1}
\frac{f'(u)^2}{f(u)f''(u)}=\frac{\beta^2+2\beta\gamma u^{\gamma}+\gamma^2u^{2\gamma}}{\beta(\beta-1)+\gamma(2\beta+\gamma-1)u^{\gamma}+\gamma^2u^{2\gamma}},
\end{equation}
and hence (A3) with $q_f=1$ holds.
Moreover, the RHS of \eqref{S2E1} is less than or equal to $1$ for large $u>0$, and hence (A6) holds.
If $2<N<10$, then $q_{JL}<q_f=1<q_S$, and hence (A5) holds.
Therefore, the conclusion of Theorem~\ref{S1T1} holds for $2<N<10$.
\subsection{Example 2}
The second example is
$$
f(u)=\chi(u) e^{au},\quad a=20,
$$
where $\chi(u):=\int_0^u\chi'(s)ds$ and
$$
\chi'(u):=
\begin{cases}
0 & \text{for}\ u\ge 4;\\
5(u-4)^4 & \text{for}\ 3\le u\le 4;\\
10-5(u-2)^4 & \text{for}\ 1\le u\le 3;\\
5u^4 & \text{for}\ 0\le u\le 1.
\end{cases}
$$
In particular, $f\in C^1[0,\infty)\cap C^2(0,\infty)$ and
$$
f(u)=\begin{cases}
u^5 e^{au} & \text{for}\ 0\le u\le 1;\\
20e^{au} & \text{for}\ u\ge 4.
\end{cases}
$$
We easily see that {\rm (A1)}, {\rm (A2)} and {\rm (A4)} hold.
See \cite[Example 2]{HM25} for details.
By direct calculation we have
$$
\frac{f'(u)^2}{f(u)f''(u)}=1,
$$
for large $u>0$, and hence (A3) with $q_f=1$ and (A6) hold.
If $2<N<10$, then $q_{JL}<q_f=1<q_S$, and hence (A5) holds.
Therefore, the conclusion of Theorem~\ref{S1T1} holds for $2<N<10$.

\subsection{Example 3}
The third example is
$$
f(u)=u^{\beta}+u^{\gamma},\quad p_S<\gamma<\beta<p_{JL}.
$$
It is easy to check that {\rm (A1)}, {\rm (A2)} and {\rm (A4)} hold.
Since
$$
\frac{f'(u)^2}{f(u)f''(u)}
=\frac{\beta^2u^{2(\beta-\gamma)}+2\beta\gamma u^{\beta-\gamma}+\gamma^2}{\beta(\beta-1)u^{2(\beta-\gamma)}+(\beta^2-\beta+\gamma^2-\gamma)u^{\beta-\gamma}+\gamma(\gamma-1)},
$$
(A3) with $q_f=\beta/(\beta-1)$ holds, and hence $p_f=\beta$.
Since $p_S<p_f<p_{JL}$, we see that $q_{JL}<q_f<q_S$, and hence (A5) holds.
Since $q_f\neq 1$, (A6) automatically holds.
Therefore, the conclusion of Theorem~\ref{S1T1} holds.

\section{Singular solution}
\begin{lemma}\label{S3L2}
Suppose that {\rm (A1)}--{\rm (A3)}.
Then,
\[
\lim_{u\to\infty}f'(u)F(u)=q_f.
\]
\end{lemma}

\begin{proof}
It is obvious that
$$
F(u)\to 0\quad \text{as}\ u\to\infty.
$$
By {\rm (A1)} and {\rm (A2)} we see that $f'(u)\to\infty$ as $u\to\infty$.
Hence
$$
\frac{1}{f'(u)}\to 0\quad\text{as}\ u\to\infty.
$$
By L'Hospital's rule we have
$$
\lim_{u\to\infty}\frac{F(u)}{1/f'(u)}=\lim_{u\to\infty}\frac{-1/f(u)}{-f''(u)/f'(u)^2}
=\lim_{u\to\infty}\frac{f'(u)^2}{f(u)f''(u)}=q_f.
$$
\end{proof}

Hereafter in this section we consider the equation
\begin{equation}\label{ODE}
u''+\frac{N-1}{r}u'+f(u)=0.
\end{equation}
We call a solution $u^*$ of equation \eqref{ODE} a singular solution if $u^*(r)\to\infty$ as $r\to 0^+$.

Let $u(r,\alpha)$, $\alpha>0$, denote a solution of the initial value problem~\eqref{ODE} with $u(0,\alpha)=\alpha$ and $u_r(0,\alpha)=0$.
Because of Lemma~\ref{S3L2} and (A1)--(A3), all assumptions of \cite[Theorem~1.1 and Lemma~2.5]{MN23}, which are stated in Proposition~\ref{S2P1}, are satisfied.
\begin{proposition}\label{S2P1}
Suppose that $N>2$ and that {\rm (A1)}--{\rm (A3)} hold.
If $q_f<q_S$, then there exists a unique positive singular solution $u^*$ of equation \eqref{ODE} for $0<r\le r_1$ with some $r_1>0$, and the regular solution $u(r,\alpha)$ satisfies
\[
u(r,\alpha)\to u^*(r)\ \ \textrm{in}\ \ C^2_{loc}(0,r_1]\ \ \textrm{as}\ \ \alpha\to\infty.
\]
Furthermore, the positive singular solution $u^*$ satisfies
\begin{equation}\label{S2P1E2}
-r^{N-1}u^{*\prime}(r)=\int_0^rf(u^*(s))s^{N-1} \, \di s
\end{equation}
for $0<r<r_1$ and
\begin{equation}\label{S2P1E3}
u^*(r)=F^{-1}\left[\frac{r^2}{2N-4q_f}(1+\theta(r))\right]\ \ \textrm{as}\ \ r\to 0^+,
\end{equation}
where $F^{-1}$ is the inverse function of $F(u)$ defined by \eqref{F} and $\theta(r)$ satisfies $\lim_{r\to 0^+}\theta(r)=0$,
\begin{equation}\label{S2P1E4}
\theta(r)\in C^1(0,r_1),\quad\text{and}\quad\lim_{r\to 0^+}r\theta'(r)=0.
\end{equation}
\end{proposition}
All statements of Proposition~\ref{S2P1} except \eqref{S2P1E4} are the same as \cite[Theorem 1.1 and Lemma 2.5]{MN23}.
We briefly prove \eqref{S2P1E4}.
Because of the uniqueness of the singular solution, $u^*$ given in Proposition~\ref{S2P1} and $u^*$ constructed in \cite[Theorem A]{M18} are identical.
Let $t:=\log r$ and $x(t):=\theta(r)=-1+(2N-4q_f)r^{-2}F[u^*(r)]$.
The solution constructed in \cite[Lemma 3.1]{M18} satisfies $x(t)$ is of class $C^1$ and $x'(t)\to 0$ as $t\to-\infty$.
Then, we obtain
$$
r\theta'(r)=x'(t)\to 0\quad\text{as}\ t\to -\infty.
$$
Thus, \eqref{S2P1E4} holds.\\

In the authors' previous paper \cite{HM25} a singular stationary solution was constructed in the case $q_f=1$.
Proofs of Lemmas~\ref{S3L3}, Proposition~\ref{S1P1} and Lemma~\ref{S10L1} in the present paper, which treat the case $q_f<q_S$, are slight modifications of those of \cite[Lemma 3.3, Proposition 1.2 and Lemma 4.2]{HM25}, respectively.
However, we show proofs for self-contained purpose.
\begin{lemma}\label{S3L3}
Suppose that $N>2 $ and that {\rm (A1)}--{\rm (A3)} hold.
Let $u^*$ be the unique positive singular solution of equation \eqref{ODE} on $0<r<r_1$ given in Proposition~{\rm \ref{S2P1}}.
For any small $\delta>0$, there exist $R_0\in (0,r_1)$ and $C>0$ such that
\begin{equation}\label{S3L3E0}
u^*(r) \le C r^{2-2q_f-2\delta}  \quad \mbox{and} \quad
|u^{*\prime}(r)| \le C r^{1-2q_f-2\delta}
\end{equation}
for $r\in (0,R_0)$.
\end{lemma}
\begin{proof}
Let $\delta>0$ be sufficiently small.
By Lemma~\ref{S3L2} we have
\[
\lim_{u\to\infty} f'(u) F(u)=q_f.
\]
There exists $M>0$ such that $f'(u) F(u) \le q_f+\delta$ for all $u\ge M$.
Then
\[
\frac{f'(u)}{f(u)} \le \frac{q_f+\delta}{f(u) F(u)} = -\frac{(q_f+\delta) F'(u)}{F(u)}
\]
for all $u\ge M$.
Integrating this inequality in $u$ over $(M, u)$, we have
\[
\log \frac{f(u)}{f(M)} \le - (q_f+\delta)  \log\frac{F(u)}{F(M)}
\]
for all $u\ge M$.
Therefore,
\begin{equation}\label{S3L3E1}
\frac{1}{f(M)F(M)^{q_f+\delta}} \le \frac{1}{f(u)F(u)^{q_f+\delta}}  = - \frac{F'(u)}{F(u)^{q_f+\delta}}
\end{equation}
for all $u\ge M$.
Since $ u^*(r) \to \infty$ as $r\to 0^+$, there exists $r_2>0$ such that
\[
u^*(r) > M \quad \mbox{for}\ r\in (0,r_2).
\]
Integrating \eqref{S3L3E1} in $u$ over $(M, u^*(r))$, we have
\[
u^*(r) < \frac{f(M)F(M)^{q_f+\delta}}{q_f-1+\delta} (F(u^*(r))^{1-q_f-\delta}- F(M)^{1-q_f-\delta}) +M
\]
for $r\in (0,r_2)$.
Since $F(u^*(r)) = (2N-4q_f)^{-1}r^2(1+o(1))$ as $r\to 0^+$ (by \eqref{S2P1E3}),
there exist $R_0\in (0, r_2)$ and $C>0$ such that
\[
u^*(r) \le C r^{2-2q_f-2\delta}  \quad \mbox{for}\ r\in (0,R_0).
\]

By \eqref{S3L3E1} we have
\begin{equation}\label{S3L3E2}
f(u^*(r)) < \frac{f(M)F(M)^{q_f+\delta}}{F(u^*(r))^{q_f+\delta}} \le C r^{-2q_f-2\delta}
\end{equation}
for $r\in (0,R_0)$.
By (\ref{S2P1E2}) we have
\[
0< -r^{N-1} u^{*\prime}(r) 
= \int_0^r f(u^*(s)) s^{N-1} \, \di s
\le C \int_0^r s^{N-1-2q_f-2\delta} \, \di s \le C r^{N-2q_f-2\delta}.
\]
Then we have 
\[
|u^{*\prime}(r)| \le C r^{1-2q_f-2\delta} \quad \mbox{for} \quad r\in (0,R_0).
\]
The proof is complete.
\end{proof}

\begin{proof}[Proof of Proposition~{\rm \ref{S1P1}}]
Let $u^*$ denote the unique positive singular solution of equation~\eqref{ODE} on $0<r<r_1$ given in Proposition~\ref{S2P1}.
We extend the domain of $u^*$ such that $u^*$ satisfies equation \eqref{ODE} for all $r>0$.

We show that $u^*(r)>0$ for $r>0$.
Let $P(u^*(r))$ be defined by
\[
P(u^*(r)):=\frac{1}{2}r^Nu^{*\prime}(r)^2+r^NF_0(u^*(r))+\frac{N-2}{2}r^{N-1}u^*(r)u^{*\prime}(r)
\]
for $r>0$,
where
\[
F_0(u):=\int_0^uf(s) \, \di s.
\]
Using (\ref{ODE}), we have
\[
\frac{d}{dr}P(u^*(r))=-\frac{N-2}{2}r^{N-1}Q(u^*(r)),
\]
where $Q$ is defined in (A4).
Integrating the above equality over $(\rho,r)$, we have
\begin{equation}\label{S1P1E5}
P(u^*(r))-P(u^*(\rho))=-\frac{N-2}{2}\int_{\rho}^rs^{N-1}Q(u^*(s)) \, \di s.
\end{equation}
Since $q_f < q_S$, we can take a small $\delta >0$ such that 
$N+2-4q_f-4\delta>0$.
By L'Hospital's rule we have
\[
\lim_{\rho\to 0^+}\left|\frac{F_0(u^*(\rho))}{\rho^{-N}}\right|
=\lim_{\rho\to 0^+}\left|\frac{f(u^*(\rho))u^{*\prime}(\rho)}{-N\rho^{-N-1}}\right|
\le\lim_{\rho\to 0^+}C\rho^{N+2-4q_f-4\delta}=0,
\]
where Lemma~\ref{S3L3} and (\ref{S3L3E2}) are used.
By the definition of $P$ we have
\begin{align*}
|P(u^*(\rho))|
&\le C\rho^{N+2-4q_f-4\delta}+|\rho^NF_0(u^*(\rho))|+C\rho^{N+2-4q_f-4\delta}\\
&\to 0\quad \mbox{as}\quad \rho\to 0^+.
\end{align*}
Assume that there exists $r_0>0$ such that $u^*(r_0)=0$.
Then, by Hopf's lemma, $u^{*\prime}(r_0)<0$.
Letting $\rho\to 0^+$ in (\ref{S1P1E5}) with $r=r_0$, we have
\[
0<P(u^*(r_0))-0=-\frac{N-2}{2}\int_0^{r_0}s^{N-1}Q(u^*(s)) \, \di s\le 0,
\]
which is a contradiction.
Thus, $r_0=\infty$, and hence $u^*(r)>0$ for $0\le r<\infty$.

Next, assume that there exists $r_0^\prime>0$ such that $u^{*\prime}(r_0^\prime) =0$.
Since $u^{*}(r_0^\prime)>0$, we have
\[
P(u^{*}(r_0^\prime)) = r_0^{\prime N} F_0(u^{*}(r_0^\prime)) >0.
\]
Therefore, we arrive at the same conclusion as the above inequality with $r_0$ replaced by $r_0^\prime$, which is a contradiction. 
Thus, $u^{*\prime}$ does not change its sign on $(0,\infty)$,
and hence $u^{*\prime}(r)<0$ for all $r>0$ since $u^{*\prime}(r) \to \infty$ as $r\to 0^+$.
We can easily see that the other properties in (\ref{SS}) are satisfied.

Hereafter, we show that $u^*\in\calL^\gamma_{\ul}(\R^N)$ for $1\le \gamma<\gamma^*$.
Since $u^{*\prime}(r)<0$ for all $r>0$, at least
$\limsup_{r\to \infty} |u^{*}(r)|<\infty$
holds.
Thus, it suffices to consider integrability only near the origin.
Let $u_n(x):=\min\{n,u^*(x)\}$ for  $n\ge 1$.
Then it is obvious that $u_n\in BUC(\RN)$.
Fix $1\le \gamma <\gamma^*$ arbitrarily.
Since $N+ \gamma(2-2q_f)>0$, we can take a sufficiently small
$\delta>0$  such that $N+\gamma(2-2q_f-2\delta)>0$.
By Lemma~\ref{S3L3} we have
\[
\int_{B(0,2)}|u^*(x)-u_n(x)|^\gamma\,\di x\le C\int_0^2r^{N+\gamma(2-2q_f-2\delta)-1}\,\di r<\infty.
\]
By the dominated convergence theorem we have
\[
\lim_{n\to\infty}\int_{B(0,2)}|u^*(x)-u_n(x)|^\gamma\,\di x=0,
\]
which indicates that $\left\|u^*-u_n\right\|_{L^\gamma_{\ul}(\RN)}\to 0$ as $n\to\infty$.
Thus, $u^*\in\calL^\gamma_{\ul}(\RN)$ for $1\le \gamma <\gamma^*$.
The proof is complete.
\end{proof}

We use the following proposition to prove Lemma~\ref{S10L1}.
\begin{proposition}[{\cite[Proposition~2.2]{MT06}}]\label{S4P2}
For $1\le p<\infty$, let $\calL^p_{ul}(\RN)$ be defined by \eqref{BUC}.
The following assertions are equivalent:
\begin{itemize}
\item[{(i)}] $u\in\calL^p_{ul}(\RN)$;
\item[{(ii)}] $\displaystyle\lim_{|y|\to 0^+}\left\|u(\,\cdot\,+y)-u(\,\cdot\,)\right\|_{L^p_{ul}(\RN)}=0$;
\item[{(iii)}] $\displaystyle\lim_{t\to 0^+}\left\|S(t)u-u\right\|_{L^p_{ul}(\RN)}=0$.
\end{itemize}
\end{proposition}

\begin{proof}[{Proof of Lemma~{\rm \ref{S10L1}}}]
Let $x \in\mathbb{R}^N \setminus \{0\}$ and $t\in(0,\infty)$.
We show that $u(x,t)=u^*(x)$ is a solution in the sense of Definition~\ref{S1D1}.
Since $u^*(x)\in C^2(\RN\setminus\{0\})$, we see by Green's identity that, for $s\in (0,t)$,
\begin{equation}
\label{S10L1E1}
\begin{split}
&\int_{B(0,\e)^c}G_{t-s}(x-y)\Delta u^*(y) \, \di y\\
& =\int_{\partial B(0,\e)^c}\left(
G_{t-s}(x-y)\frac{\partial}{\partial\nu_y}u^*(y)-u^*(y)\frac{\partial}{\partial\nu_y}G_{t-s}(x-y)\right) \,\di \sigma(y)\\
&\qquad+\int_{B(0,\e)^c}\Delta_yG_{t-s}(x-y)u^*(y) \,\di y\\
&=:I_1+I_2.
\end{split}
\end{equation}
By Lemma~\ref{S3L3} we have
\begin{align}\label{S10L1E2}
|I_1|
&\le C\left|\left.\frac{\partial}{\partial\nu}u^*(y)\right|_{|y|=\e}\right|\e^{N-1}
+C\left|\left.u^*(y)\right|_{|y|=\e}\right|\e^{N-1}\nonumber\\
&\le C\e^{N-2q_f-2\delta}+C\e^{N+1-2q_f-2\delta}
\to 0\ \ \textrm{as}\ \ \e\to 0^+,
\end{align}
since $N-2q_f-2\delta>0$ and $N+1-2q_f-2\delta>0$ for small $\delta>0$.
Integrating (\ref{S10L1E1}) with respect to $s$ over $(0,t-\delta)$ , we have
\begin{equation*}
\begin{split}
&\int_0^{t-\delta}\int_{B(0,\e)^c}G_{t-s}(x-y)\Delta u^*(y) \, \di y\di s\\
&\qquad=\int_0^{t-\delta}I_1\,\di s-\int_0^{t-\delta}\int_{B(0,\e)^c}\partial_sG_{t-s}(x-y)u^*(y)\,\di y\di s\\
&\qquad=\int_0^{t-\delta}I_1\,\di s-\int_{B(0,\e)^c}\int_0^{t-\delta}\partial_sG_{t-s}(x-y)u^*(y)\,\di s\di y\\
&\qquad=\int_0^{t-\delta}I_1\,\di s-\int_{B(0,\e)^c}\left(G_{\delta}(x-y)u^*(y)-G_t(x-y)u^*(y)\right)\,\di y,
\end{split}
\end{equation*}
where we used $\Delta_yG_{t-s}(x-y)=-\partial_sG_{t-s}(x-y)$.
Letting $\e\to 0^+$, by the dominated convergence theorem and (\ref{S10L1E2}) we have
\begin{equation}\label{S10L1E3}
\lim_{\e\to 0^+}\int_0^{t-\delta}\int_{B(0,\e)^c}G_{t-s}(x-y)\Delta u^*(y) \, \di y\di s
=-S(\delta)u^*+S(t)u^*.
\end{equation}
By the monotone convergence theorem we have
\begin{equation}\label{S10L1E4}
\begin{split}
&\lim_{\e\to 0^+}\int_0^{t-\delta}\int_{B(0,\e)^c}G_{t-s}(x-y)f(u^*(y))\,\di y\di s\\
&\qquad=\lim_{\e\to 0^+}\int_0^{t-\delta}\int_{\RN}G_{t-s}(x-y)f(u^*(y))\chi_{B(0,\e)^c}(y)\,\di y\di s\\
&\qquad=\int_0^{t-\delta}\int_{\RN}G_{t-s}(x-y)f(u^*(y))\,\di y\di s.
\end{split}
\end{equation}
The solution $u^*$ satisfies $\Delta u^*+f(u^*)=0$ in the classical sense for $x\in B(0,\e)^c$.
Hence,
\begin{equation*}
\begin{split}
&\int_0^{t-\delta}\int_{B(0,\e)^c}G_{t-s}(x-y)\Delta u(y)\,\di y\di s\\
&\qquad\qquad+\int_0^{t-\delta}\int_{B(0,\e)^c}G_{t-s}(x-y)f(u^*(y))\,\di y\di s=0.
\end{split}
\end{equation*}
Letting $\e\to 0^+$, by (\ref{S10L1E3}) and (\ref{S10L1E4}) we have
\begin{equation}\label{S10L1E5}
-S(\delta)u^*+S(t)u^*+\int_0^{t-\delta}S(t-s)f(u^*) \, \di s=0.
\end{equation}
Since $u^*\in\calL^\gamma_{\ul}(\RN)$ (by Proposition~\ref{S1P1}), by Proposition~\ref{S4P2} with $p=\gamma$ we have
\begin{equation}\label{S10L1E6}
\lim_{\delta\to 0^+}\left\|S(\delta)u^*-u^*\right\|_{L^{\gamma}_{\ul}(\RN)}=0.
\end{equation}
By the monotone convergence theorem we have
\begin{equation}\label{S10L1E7}
\begin{split}
\lim_{\delta\to 0^+}\int_0^{t-\delta}S(t-s)f(u^*)\,\di s
&=\lim_{\delta\to 0^+}\int_0^t\left(S(t-s)f(u^*)\right)\chi_{(0,t-\delta)}(s)\, \di s\\
&=\int_0^tS(t-s)f(u^*)\,\di s.
\end{split}
\end{equation}
By \eqref{S10L1E7}, \eqref{S10L1E6}, and \eqref{S10L1E5} we see that as $\delta\to 0^+$,
\[
u^*=S(t)u^*+\int_0^tS(t-s)f(u^*)\,\di s
\]
for all $(x,t)\in(\R^N\setminus\{0\})\times (0,\infty)$.
Thus, $u^*$ is a solution in the sense of Definition~\ref{S1D1}.
The proof is complete.
\end{proof}

\section{Transformations of self-similar solutions}
In order to treat two canonical nonlinearities in a unified way we define
$$
f_q(U):=\begin{cases}
U^p & \text{if}\ q>1,\\
e^U & \text{if}\ q=1,
\end{cases}
\qquad
F_q(U):=\int_U^{\infty}\frac{\di s}{f_q(s)}=
\begin{cases}
\displaystyle{\frac{1}{p-1}U^{-p+1}} & \text{if}\ q>1,\\
e^{-U} & \text{if}\ q=1,
\end{cases}
$$
where $p:=q/(q-1)$.
We consider the canonical equation
\begin{equation}\label{S4E1}
\partial_tU-\Delta U=f_q(U).
\end{equation}
For $\lambda>0$, $U_{\lambda}(x,t)$ is defined by the relation
\begin{equation}\label{S4E1+}
F_q[U_{\lambda}(x,t)]=\lambda^{-2}F_q[U(\lambda x,\lambda^2t)].
\end{equation}
If $U(x,t)$ satisfies \eqref{S4E1}, then $U_{\lambda}$ also satisfies \eqref{S4E1}.
We can easily check that
$$
U_{\lambda}(x,t)=\begin{cases}
\lambda^{\frac{2}{p-1}}U(\lambda x,\lambda^2t) & \text{if}\ q>1,\\
U(\lambda x,\lambda^2t)-2\log\lambda & \text{if}\ q=1.
\end{cases}
$$
To derive a forward self-similar solution we put $\lambda=1/\sqrt{t}$ into \eqref{S4E1+}.
We redefine $U$ by
\begin{equation}\label{S4E4}
F_q[U(x,t)]=tF_q\left[\varphi(\xi)\right],\quad \xi:=\frac{x}{\sqrt{t}}.
\end{equation}
Since $U$ satisfies \eqref{S4E1}, $\varphi(\xi)$ satisfies
\begin{equation}\label{S4E2}
\Delta\varphi+\frac{1}{2}\xi\cdot\nabla\varphi+f_q(\varphi)F_q(\varphi)+f_q(\varphi)=0.
\end{equation}
Specifically, $\varphi$ satisfies
\begin{equation}\label{S4E3}
\begin{cases}
\displaystyle{\Delta\varphi+\frac{1}{2}\xi\cdot\nabla\varphi+\frac{\varphi}{p-1}+\varphi^p=0} & \text{if}\ q>1,\\
\displaystyle{\Delta\varphi+\frac{1}{2}\xi\cdot\nabla\varphi+1+e^{\varphi}=0} & \text{if}\ q=1.
\end{cases}
\end{equation}
In summary, if $\varphi$ is a solution of equation~\eqref{S4E2}, then $U$ is a self-similar solution of equation~\eqref{S4E1}.

Self-similar solutions are important in the study of equation~\eqref{S4E1}, {\it e.g.}, existence of global-in-time solutions, asymptotic behaviors and nonuniqueness.
Then, various properties of solutions of equation~\eqref{S4E3} have been studied in numerous papers.
We recall known results about the intersection number of radial regular and singular solutions of equation~\eqref{S4E3}.
Here, we state known results in a unified way, using \eqref{S4E2}.
Equation~\eqref{S4E1} has an explicit singular solution
$$
U^*(x)=F_q^{-1}\left[\frac{|x|^2}{2N-4q}\right]=\begin{cases}
L|x|^{-\frac{2}{p-1}} & \text{if}\ q>1,\\
-2\log|x|+\log 2(N-2) & \text{if}\ q=1,
\end{cases}
$$
where $L$ is defined by \eqref{L} and $F_q^{-1}$ denotes the inverse function of $F_q$, {\it i.e.,}
$$
F_q^{-1}(u)=
\begin{cases}
\displaystyle{\frac{1}{\{(p-1)u\}^{1/(p-1)}}} & \text{if}\ q>1,\\
-\log u & \text{if}\ q=1.
\end{cases}
$$
The function
$$
\varphi^*(\xi):=F_q^{-1}\left[\frac{|\xi|^2}{2N-4q}\right],
$$
which is the same function as $U^*$, is a solution of equation~\eqref{S4E2}.
Let $\eta:=|\xi|$, and let $\varphi(\eta,\alpha)$ denote the solution of
$$
\begin{cases}
\displaystyle{\varphi''+\frac{N-1}{\eta}\varphi'+\frac{1}{2}\eta\varphi'+f_q(\varphi)F_q(\varphi)+f_q(\varphi)=0}, & \text{for}\ \eta>0,\\
\varphi(0,\alpha)=\alpha>0,\ \varphi_{\eta}(0,\alpha)=0.
\end{cases}
$$
\begin{proposition}\label{S4P1}
Assume that $1\le q<q_S$.
Then $\varphi(\eta,\alpha)$ is defined for $\eta\ge 0$ and $\varphi_{\eta}(\eta,\alpha)<0$ for $\eta>0$.
There exists $\alpha_0>0$ such that $\varphi(\eta,\alpha_0)$ and $\varphi^*(\eta)$ has the first intersection point $\eta_0$.
Specifically, there exists $\delta>0$ such that
$$
\begin{cases}
\varphi(\eta,\alpha_0)<\varphi^*(\eta)& \text{for}\ 0<\eta<\eta_0,\\
\varphi(\eta_0,\alpha_0)=\varphi^*(\eta_0)\ \text{and}\ \varphi_{\eta}(\eta_0,\alpha_0)>\varphi^{*\prime}(\eta_0),\\
\varphi(\eta,\alpha_0)>\varphi^*(\eta)& \text{for}\ \eta_0<\eta<\eta_0+\delta.
\end{cases}
$$
\end{proposition}
The proof can be found in \cite[Theorem 1.4]{SW03} for $1<q<q_S$ and in \cite[Proposition~4.2]{F18} for $q=1$.
The detailed solution structure of the case $1<q<q_S$ can be found in \cite{N06}.
\\

Hereafter, we consider problem~\eqref{S1E1}, and let $q_f$ be the exponent defined in (A3).
If $q_f=1$, then we define $q:=1$.
If $q_f>1$, then we define $q:=q_f+\e$, where $\e>0$ is determined later.
We define $\bu(x,t)$ by
\begin{equation}\label{S4E4+}
\bu(x,t):=F^{-1}[F_q[U(x,t)]]=F^{-1}\left[tF_q\left[\varphi\left(\frac{|x|}{\sqrt{t}}\right)\right]\right],
\end{equation}
where $U$ is a self-similar solution of equation~\eqref{S4E1} defined by \eqref{S4E4}, $\varphi(\eta):=\varphi(\eta,\alpha_0)$ and $\alpha_0$ is given in Proposition~\ref{S4P1}.
It is obvious from the definition that $\|\overline{u}(\cdot, t)\|_{L^\infty(\mathbb{R}^N)} < \infty$ for every $t>0$.
Since
\begin{equation}\label{S4E5}
F[\bu(x,t)]=F_q[U(x,t)],
\end{equation}
$\bu$ is a transformation of a self-similar solution $U$ of \eqref{S4E1}.
Since $U$ satisfies \eqref{S4E1}, by \eqref{S4E5} we see that $\bu$ is a solution of the equation
\begin{equation}\label{S4E5+}
\partial_t\bu-\Delta \bu=f(\bu)+\frac{q-f'(\bu)F(\bu)}{f(\bu)F(\bu)}|\nabla \bu|^2.
\end{equation}
The transformation \eqref{S4E5} and the equation \eqref{S4E5+} were found in Fujishima-Ioku~\cite{FI18}.
If $q_f=1$, then by (A6) we integrate $1/f(\bu)\le f''(\bu)/f'(\bu)^2$ over $[\bu,\infty)$.
Then, $F(\bu)\le /f'(\bu)$ for large $\bu>0$.
Hence, $f'(\bu)F(\bu)\le q(=1)$ for large $\bu>0$.
If $q_f>1$, then by Lemma~\ref{S3L2} we see that $f'(\bu)F(\bu)\to q_f$ $(\bu\to\infty)$, and hence $f'(\bu)F(\bu)\le q(=q_f+\e)$ for large $\bu>0$.
In both cases $q-f'(\bu)F(\bu)\ge 0$ for large $\bu>0$, and hence by \eqref{S4E5+} we have
\begin{equation}\label{S4E5++}
\partial_t\bu\ge \Delta\bu+f(\bu)
\end{equation}
for large $\bu>0$,
Thus, $\bu$ is a supersolution of problem~\eqref{S1E1} in the classical sense if $\bu$ is large.

We study an intersection point of $u^*(x)$, which is a positive singular stationary solution of \eqref{S1E1} given by Proposition~\ref{S1P1}, and $\bu(x,t)$.
\begin{lemma}\label{S4L1}
There exists a small $t_0>0$ such that $\bu(r,t)$ and $u^*(r)$ have the first intersection point $r=r(t)\in C^1(0,t_0)$ such that $r(t)\to 0$ as $t\to 0^+$.
Specifically, there exists a positive continuous function $\delta(t)\in C(0,t_0)$ such that
\begin{equation}\label{S4L1E1}
\begin{cases}
\bu(r,t)<u^*(r) & \text{for}\ 0<r<r(t),\\
\bu(r(t),t)=u^*(r(t))\ \text{and}\ \bu_r(r(t),t)>u^{*\prime}(r(t)),\\
\bu(r,t)>u^*(r) & \text{for}\ r(t)<r<r(t)+\delta(t).
\end{cases}
\end{equation}
\end{lemma}

\begin{proof}
First, we prove the assertion, assuming that $q=q_f$ ($\e=0$).
In order to find an intersection point we consider the equation $\bu=u^*$, namely
\begin{equation}\label{S4L1E1+}
F^{-1}\left[tF_q\left[\varphi\left(\frac{r}{\sqrt{t}},\alpha_0\right)\right]\right]=
F^{-1}\left[\frac{r^2}{2N-4q_f}(1+\theta(r))\right].
\end{equation}
Here, by Proposition~\ref{S2P1} we see that as $r\to 0^+$,
\begin{equation}\label{S4L1E2}
\theta(r)\to 0\quad\text{and}\quad r\theta'(r)\to 0.
\end{equation}
Let $\eta:=r/\sqrt{t}$. Then, by \eqref{S4L1E1+} we have
$$
F_q[\varphi(\eta,\alpha_0)]=\frac{\eta^2}{2N-4q_f}(1+\theta(r)).
$$
Therefore,
\begin{equation}\label{S4L1E3}
\varphi(\eta,\alpha_0)=F_q^{-1}\left[\frac{\eta^2+\eta^2\theta(r)}{2N-4q_f}\right].
\end{equation}
Putting $r=\sqrt{t}\eta$ into \eqref{S4L1E2}, we have
\begin{equation}\label{S4L1E4}
\theta(\sqrt{t}\eta)\to 0\quad\text{in}\ C_{\eta,loc}(0,\infty)\quad\text{as}\ t\to 0^+.
\end{equation}
Because of Proposition~\ref{S4P1}, by \eqref{S4L1E4} we see that \eqref{S4L1E3} has the first intersection $\eta_*(t)$ near $\eta_0$ for small $t>0$ such that an intersection point does not exist on $0<\eta<\eta_*(t)$.
We define $r(t):=\sqrt{t}\eta_*(t)$.
By the uniqueness of the solution of equation~\eqref{S4E2} we see that
\begin{equation}\label{S4L1E5}
\varphi'(\eta_0,\alpha_0)\neq\frac{d}{d\eta}\left.F_q^{-1}\left[\frac{\eta^2}{2N-4q_f}\right]\right|_{\eta=\eta_0}.
\end{equation}
Using \eqref{S4L1E2}, we have
\begin{align}
\frac{d}{d\eta}\left(\eta^2\theta(\sqrt{t}\eta)\right)
&=2\eta\theta(\sqrt{t}\eta)+\eta(\sqrt{t}\eta)\theta'(\sqrt{t}\eta)\nonumber\\
&\to 0\quad\text{in}\ C_{\eta,loc}(0,\infty)\quad\text{as}\ t\to 0^+.\label{S4L1E6}
\end{align}
It follows from \eqref{S4L1E4} and \eqref{S4L1E6} that the LHS of \eqref{S4L1E3} converges to $\varphi^*(\eta)$ in $C^1_{loc}(0,\infty)$ as $t\to 0^+$.
By \eqref{S4L1E5} we have
$$
\varphi'(\eta_*(t),\alpha_0)
\neq
\frac{d}{d\eta}\left. F_q^{-1}\left[\frac{\eta^2}{2N-4q_f}(1+\theta(\sqrt{t}\eta))\right]\right|_{\eta=\eta_*(t)}
$$
for small $t>0$.
Applying the implicit function theorem to the function in $(\eta,t)$ defined by
$$
\varphi(\eta,\alpha_0)-F_q^{-1}\left[\frac{\eta^2}{2N-4q_f}(1+\theta(\sqrt{t}\eta))\right]=0,
$$
we see that $\eta_*(t)$, which is a function in $t$, is of class $C^1$.
Thus, $r(t)=\sqrt{t}\eta_*(t)$ is of class $C^1$.
Since $\eta_*(t)\to \eta_0$ as $t\to 0^+$, we see that $r(t)\to 0$ as $t\to 0^+$.

Then, it is obvious that there exists $t_0>0$ such that \eqref{S4L1E1+} has the first intersection point $r(t)$ satisfying \eqref{S4L1E1} for $0<t<t_0$.
The proof for the case $q_f=1$ is complete, since $q=q_f$.

Second, we consider the case $q_f>1$.
Let us recall that $q=q_f+\e$.
Let $X:=\{q\in\R;\ q\ge 1\}$.
For each compact interval $I\subset (0,\infty)$, the mappings $q\mapsto \varphi(\eta,\alpha_0)$ and $q\mapsto F^{-1}_q[\eta^2/(2N-4q_f)]$ are in $C(X,C^1_{\eta}(I))$.
Since \eqref{S4L1E5} with $q=q_f$ holds, we can choose a small $\e>0$ such that $q_f+\e<q_S$ and
$$
\varphi(\eta,\alpha_0)=F_q^{-1}\left[\frac{\eta^2}{2N-4q_f}\right]
$$
with $q=q_f+\e$ has the first intersection for small $\e>0$, which can be transversal, denoted by $\tilde{\eta}_0$.
Because of \eqref{S4L1E4} and \eqref{S4L1E6},
$$
\varphi(\eta,\alpha_0)=F_q^{-1}\left[\frac{\eta^2+\eta^2\theta(\sqrt{t}\eta)}{2N-4q_f}\right]
$$
with $q=q_f+\e$ also has the first intersection point for small $t>0$, which can also be transversal, denoted by $\tilde{\eta}_*(t)$.
Note that $\theta$ depends on $q_f$ but it does not depend on $q$.
We see that $\tilde{\eta}_*(t)\to\tilde{\eta}_0$ as $t\to 0$.
Let $r(t):=\sqrt{t}\tilde{\eta}_*(t)$.
Since the first intersection point $\tilde{\eta}_*(t)$ is transversal, in a similar way to before we see that there exists $t_0>0$ such that \eqref{S4L1E1+} has the first intersection point $r(t)$ satisfying \eqref{S4L1E1} for $0<t<t_0$.
The proof for the case $q_f>1$ is complete.
\end{proof}

Now, we define a function $v$ by
\begin{equation}\label{S4E7}
v(x,0)=u^*(|x|),\quad
v(x,t):=\begin{cases}
\bu(|x|,t) & \text{if}\ 0\le |x|<r(t)\ \text{and}\ 0<t<t_0,\\
u^*(|x|) & \text{if}\ |x|\ge r(t)\ \text{and}\ 0<t<t_0,
\end{cases}
\end{equation}
where $r(t)$ and $t_0$ are given in Lemma~\ref{S4L1}.
By the same argument as in the proof of Proposition~\ref{S1P1},
we see that for every $0\le t <t_0$, $v(t) \in \mathcal{L}_{ul}^\gamma(\mathbb{R}^N)$ for $1\le \gamma <\gamma^*$, where $\gamma_*$ is as in Proposition~\ref{S1P1}.
Since $r(t)\to 0$ as $t\to 0^+$, we see that $u^*(r(t))\to\infty$ as $t\to 0^+$.
Then, $\bu(x,t)$ is large in $\{x\in\RN;\ |x|\le r(t)\}$ for small $t>0$, and hence $\bu$ becomes a supersolution in the classical sense near the origin for small $t>0$ as mentioned after \eqref{S4E5++}.

We use the following proposition in this section.
\begin{proposition}[{\cite[Corollary~3.1]{MT06}}]\label{S4P3}
Let $1\le p\le q\le\infty$.
Then there exists $C_1=C_1(N,p,q)>0$ such that
$$
\left\|S(t)u\right\|_{L^q_{ul}(\RN)}\le C_1\left(t^{-\frac{N}{2}\left(\frac{1}{p}-\frac{1}{q}\right)}+1\right)
\left\|u\right\|_{L^p_{ul}(\RN)}
$$
for $t>0$ and $u\in L^p_{ul}(\RN)$.
\end{proposition}

Then, the following lemma plays a crucial role in this paper
\begin{lemma}\label{S4L2}
Let $v$ be defined by \eqref{S4E7}.
Then, $v$ is a supersolution of problem~\eqref{S1E1} in $\mathbb{R}^N\times[0,t_0)$ in the sense of Definition~{\rm \ref{S1D1}}.
\end{lemma}

\begin{proof}
Let $r(t)$ and $t_0$ be given in Lemma~\ref{S4L1}.
Let $0<s<t_0$.
$B_s$ denotes $B(0,r(s))$ for simplicity, and $B_s^c:=\RN\setminus B_s$.
By Green's formula we have
\begin{multline}\label{S4L2E1}
\int_{B_s}G_{t-s}(x-y)\Delta\bu(y)\di y\\
=\int_{\partial B_s}\left(G_{t-s}(x-y)\frac{\partial\bu}{\partial{\nu_y}}-\bu\frac{\partial}{\partial\nu_y}G_{t-s}(x-y)\right)\di S_y
+\int_{B_s}\Delta_yG_{t-s}(x-y)\bu(y)\di y,
\end{multline}
\begin{multline}\label{S4L2E2}
\int_{B_s^c}G_{t-s}(x-y)\Delta u^*(y)\di y\\
=\int_{\partial B_s^c}\left(G_{t-s}(x-y)\frac{\partial u^*}{\partial{\nu_y}}-u^*\frac{\partial}{\partial\nu_y}G_{t-s}(x-y)\right)\di S_y
+\int_{B_s^c}\Delta_yG_{t-s}(x-y)u^*(y)\di y,
\end{multline}
where $\partial/\partial{\nu_y}$ denotes the outer unit normal derivative.
Adding \eqref{S4L2E1} and \eqref{S4L2E2}, we have
\begin{align}
{\int_{B_s}G_{t-s}(x-y)\Delta v(y)\di y}
& { +\int_{B^c_s}G_{t-s}(x-y)\Delta v(y)\di y}\nonumber\\
&=\int_{\partial B_s}G_{t-s}(x-y)\frac{\partial\bu}{\partial{\nu_y}}\di S_y
+\int_{\partial B_s^c}G_{t-s}(x-y)\frac{\partial u^*}{\partial{\nu_y}}\di S_y\nonumber\\
&\quad -\int_{\partial B_s}\bu\frac{\partial}{\partial{\nu_y}}G_{t-s}(x-y)\di S_y-\int_{\partial B_s^c}u^*\frac{\partial}{\partial{\nu_y}}G_{t-s}(x-y)\di S_y\nonumber\\
&\qquad +\int_{B_s}\Delta_yG_{t-s}(x-y)\bu(y)\di y+\int_{B_s^c}\Delta_yG_{t-s}(x-y)u^*(y)\di y\nonumber\\
&\ge\int_{\RN}\Delta_yG_{t-s}(x-y)v(y)\di y,\label{S4L2E3}
\end{align}
where we used
\begin{multline*}
\int_{\partial B_s}G_{t-s}(x-y)\frac{\partial\bu}{\partial{\nu_y}}\di S_y
+\int_{\partial B_s^c}G_{t-s}(x-y)\frac{\partial u^*}{\partial{\nu_y}}\di S_y\\
=\int_{\partial B_s}G_{t-s}(x-y)\left(\frac{\partial\bu}{\partial{\nu_y}}-\frac{\partial u^*}{\partial{\nu_y}}\right)\di S_y\ge 0\ \textrm{by \eqref{S4L1E1}},
\end{multline*}
\begin{multline*}
\int_{\partial B_s}\bu\frac{\partial}{\partial{\nu_y}}G_{t-s}(x-y)\di S_y
+\int_{\partial B_s^c}u^*\frac{\partial}{\partial{\nu_y}}G_{t-s}(x-y)\di S_y\\
=\int_{\partial B_s}(\bu-u^*)\frac{\partial}{\partial{\nu_y}}G_{t-s}(x-y)\di S_y=0\ \textrm{by \eqref{S4L1E1}}.
\end{multline*}
Let $\delta>0$ be small.
Integrating \eqref{S4L2E3} over $(\delta,t-\delta)$, we have
\begin{multline}\label{S4L2E4}
{\int_{\delta}^{t-\delta}\int_{B_s}G_{t-s}(x-y)\Delta v \, \di y\di s
+\int_{\delta}^{t-\delta}\int_{B_s^c}G_{t-s}(x-y)\Delta v \, \di y\di s}\\
\ge -\int_{\delta}^{t-\delta}\int_{\RN}\partial_sG_{t-s}(x-y)v \, \di y\di s,
\end{multline}
where we used $\Delta_yG_{t-s}(x-y)=-\partial_s G_{t-s}(x-y)$.

Let
$$
\B_{\delta}:=\left\{(y,s)\in\bigcup_{\delta<\tau<t-\delta}B_{\tau}\times\{\tau\}\right\}
\quad\text{and}\quad
\B_{\delta}^c:=\left\{(y,s)\in\bigcup_{\delta<\tau<t-\delta}B_{\tau}^c\times\{\tau\}\right\}
$$
Then, $\bu$ and $u^*$ satisfy the following in the classical sense:
$$
\partial_s \bu-\Delta_y\bu=f(\bu)\quad \text{in}\ \B_{\delta}\quad \text{and}\quad
\partial_su^*-\Delta_y u^*=f(u^*)\quad \text{in}\ \B_{\delta}^c.
$$
Integrating them, we have
\begin{multline}\label{S4L2E8}
\int_{\delta}^{t-\delta}\int_{B_s}G_{t-s}(x-y)\bu_s\di y\di s\\
=\int_{\delta}^{t-\delta}\int_{B_s}G_{t-s}(x-y)\Delta\bu \di y\di s
+\int_{\delta}^{t-\delta}\int_{B_s}G_{t-s}(x-y)f(\bu)\di y\di s,
\end{multline}
\begin{multline}\label{S4L2E9}
\int_{\delta}^{t-\delta}\int_{B_s^c}G_{t-s}(x-y)u^*_s\di y\di s\\
=\int_{\delta}^{t-\delta}\int_{B_s^c}G_{t-s}(x-y)\Delta u^* \di y\di s
+\int_{\delta}^{t-\delta}\int_{B_s^c}G_{t-s}(x-y)f(u^*)\di y\di s.
\end{multline}
We consider the LHSs of \eqref{S4L2E8} and \eqref{S4L2E9}.
By Fubini's theorem and integration by parts we have
\begin{align}
\int_{\delta}^{t-\delta}\int_{B_s}&G_{t-s}(x-y)\bu_s\di y\di s=\int_{\B_{\delta}}G_{t-s}(x-y)\bu_s\di (y,s)\nonumber\\
&=\int_{\partial\B_{\delta}}G_{t-s}(x-y)\bu\nu_1^{(s)}\di S_{{(y,s)}}-\int_{\B_{\delta}}\partial_sG_{t-s}(x-y)\bu \di (y,s)\nonumber\\
&=\int_{\partial\B_{\delta}}G_{t-s}(x-y)\bu\nu_1^{(s)}\di S_{{(y,s)}}-\int_{\delta}^{t-\delta}\int_{B_s}\partial_sG_{t-s}(x-y)\bu \di y\di s,\label{S4L2E5}
\end{align}
\begin{align}
\int_{\delta}^{t-\delta}\int_{B_s^c}&G_{t-s}(x-y)u^*_s\di y\di s=\int_{\B_{\delta}^c}G_{t-s}(x-y)u_s^*\di (y,s)\nonumber\\
&=\int_{\partial\B_{\delta}^c}G_{t-s}(x-y)u^*\nu_2^{(s)}\di S_{{(y,s)}}-\int_{\B_{\delta}^c}\partial_sG_{t-s}(x-y)u^* \di (y,s)\nonumber\\
&=\int_{\partial\B_{\delta}^c}G_{t-s}(x-y)u^*\nu_2^{(s)}\di S_{{(y,s)}}-\int_{\delta}^{t-\delta}\int_{B_s^c}\partial_sG_{t-s}(x-y)u^* \di y\di s,\label{S4L2E6}
\end{align}
where $\nu_1^{(s)}$ (resp.~$\nu_2^{(s)}$) denotes the $s$ element of the outer unit normal vector on $\partial\B_{\delta}$ (resp.~$\partial\B^c_{\delta}$) and in the integration by parts in \eqref{S4L2E6} we used the fact that $G_{t-s}(x-y)$ is exponentially small for large $|y|$.
Note that $\bigcup_{\delta<\tau<t-\delta}\partial B_{\tau}\times\{\tau\}(\subset\partial\B_{\delta}\cap\partial\B_{\delta}^c)$ is a $C^1$ surface (Lemma~\ref{S4L1}) and hence the boundary integrals in \eqref{S4L2E5} and \eqref{S4L2E6} are well defined.
Adding \eqref{S4L2E5} and \eqref{S4L2E6}, we have
\begin{align}
{\int_{\delta}^{t-\delta}\int_{B_s}G_{t-s}(x-y)v_s }& {\di y\di s
+\int_{\delta}^{t-\delta}\int_{B_s^c}G_{t-s}(x-y)v_s\di y\di s}
\nonumber\\
&=\left.\int_{\RN}G_{t-s}(x-y)v\di y\right|_{s=t-\delta}-\left.\int_{\RN}G_{t-s}(x-y)v\di y\right|_{s=\delta}\nonumber\\
&\qquad-\int_{\delta}^{t-\delta}\int_{\RN}\partial_s G_{t-s}(x-y)v\di y\di s\nonumber\\
&=S(\delta)v(t-\delta)-S(t-\delta)v(\delta)-\int_{\delta}^{t-\delta}\int_{\RN}\partial_s G_{t-s}(x-y)v\di y\di s.\label{S4L2E7}
\end{align}
Here we used $\nu_1^{(s)}+\nu_2^{(s)}=0$ and
$$
\int_{\partial\B_{\delta}\cap\Gamma}G_{t-s}(x-y)\bu\nu_1^{(s)}\di S_{{(y,s)}}+\int_{\partial\B_s^c\cap\Gamma}G_{t-s}(x-y)u^*\nu_2^{(s)} \di S_{{(y,s)}}=0,
$$
where $\Gamma:= \left\{(y,s)\in \bigcup_{\delta<\tau<t-\delta}\partial B_{\tau}\times\{\tau\}\right\}$.
Adding \eqref{S4L2E8} and \eqref{S4L2E9}, we have
\begin{multline}\label{S4L2E10}
{\int_{\delta}^{t-\delta}\int_{B_s}G_{t-s}(x-y)v_s\di y\di s
+\int_{\delta}^{t-\delta}\int_{B^c_s}G_{t-s}(x-y)v_s\di y\di s}
\\
={\int_{\delta}^{t-\delta}\int_{B_s}G_{t-s}(x-y)\Delta v\di y\di s+
\int_{\delta}^{t-\delta}\int_{B^c_s}G_{t-s}(x-y)\Delta v\di y\di s}\\
+\int_{\delta}^{t-\delta}{\int_{\RN}}G_{t-s}(x-y)f(v)\di y\di s.
\end{multline}
Using \eqref{S4L2E4} and \eqref{S4L2E7}, by \eqref{S4L2E10} we have
\begin{multline}\label{S4L2E11}
S(\delta)v(t-\delta)-S(t-\delta)v(\delta)-\int_{\delta}^{t-\delta}\int_{\RN}\partial_sG_{t-s}(x-y)v\di y\di s\\
\ge-\int_{\delta}^{t-\delta}\int_{\RN}\partial_sG_{t-s}(x-y)v\di y\di s+\int_{\delta}^{t-\delta}\int_{\RN}G_{t-s}(x-y)f(v)\di y\di s.
\end{multline}
By the monotone convergence theorem we have
$$
\lim_{\delta\to {0^+}}\int_{\delta}^{t-\delta}\int_{\RN}G_{t-s}(x-y)f(v)\di  y\di s
=\int_0^t\int_{\RN}G_{t-s}(x-y)f(v)\di y\di s
$$
for a.a.~$(x,y)\in\RN\times (0,t_0)$.
Using Proposition~\ref{S4P3}, for $1\le \gamma<\gamma^*$, we have
\begin{align*}
&\left\|S(\delta)v(t-{\delta})-v(t)\right\|_{L^{{\gamma}}_{ul}(\RN)}\\
&\qquad\le\left\|S(\delta)v(t-\delta)-S(\delta)v(t)\right\|_{L^{{\gamma}}_{ul}(\RN)}+\left\|S(\delta)v(t)-v(t)\right\|_{L^{{\gamma}}_{ul}(\RN)}\\
&\qquad\le C_0\left\| v(t-\delta)-v(t)\right\|_{L^{{\gamma}}_{ul}(\RN)}+\left\|S(\delta)v(t)-v(t)\right\|_{L^{{\gamma}}_{ul}(\RN)}\\
&\qquad\to 0\ \text{as}\ \delta\to 0^+,
\end{align*}
where $\gamma^*$ is as in Proposition~\ref{S1P1} and we used $\lim_{\delta\to {0^+}}\int_{B(0,2)}|v(t-\delta)-v(t)|^\gamma \, \di x=0$ and Proposition~\ref{S4P2} with $v(t)\in\calL^{{\gamma}}_{ul}(\RN)$.
{Similarly,} by Proposition~\ref{S4P3} we have
\begin{align*}
&\left\|S(t-\delta)v(\delta)-S(t)u^*\right\|_{L^{{\gamma}}_{ul}(\RN)}\\
&\qquad=\left\|S(t-\delta)v(\delta)-S(t-\delta)u^*\right\|_{L^{{\gamma}}_{ul}(\RN)}+\left\|S(t-\delta)u^*-S(t)u^*\right\|_{L^{{\gamma}}_{ul}(\RN)}\\
&\qquad\le 2C_1\left\|v(\delta)-u^*\right\|_{L^{{\gamma}}_{ul}(\RN)}+\left\|S(t-\delta)u^*-S(t)u^*\right\|_{L^{{\gamma}}_{ul}(\RN)}\\
&\qquad \to 0\ \text{as}\ \delta\to 0^+,
\end{align*}
where we used $\lim_{\delta\to 0}\int_{B(0,2)}|v(\delta)-u^*|^\gamma \, \di
x=0$.
We let $\delta\to 0^+$ in \eqref{S4L2E11}.
Then, we obtain
$$
v(t)-S(t)u^*\ge \int_0^t S(t-s)f(v(s))\di s
$$
for a.a.~$(x,t)\in\RN\times (0,t_0)$.
Thus, $v(t)$ is a supersolution of problem~\eqref{S1E1} in $\mathbb{R}^N\times[0,t_0)$ in the sense of Definition~\ref{S1D1}.
\end{proof}

\section{Nonuniqueness (supercritical)}
\begin{lemma}\label{S5L1}
Let $\bw(t)$ be a supersolution of problem~\eqref{S1E1} in $\mathbb{R}^N\times[0,T)$ in the sense of Definition~{\rm \ref{S1D1}}, where  $T>0$.
If $0\le w_0\le \bw(0)$ a.e.~in $\RN$, then problem~\eqref{S1E1} with initial data $w_0$ has a positive solution $w(t)$ such that $w(t)\le \bw(t)$ a.e.~in $\RN\times[0,T)$.
\end{lemma}

\begin{proof}
We define
\begin{equation}\label{S5L1E1}
w_1=0,\quad w_k(t)=S(t)w_0+\int_0^tS(t-s)f(w_{k-1}(s))\di s
\end{equation}
for $k=2,3,\ldots$.
By induction we can easily see that
$$
0=w_1\le w_2\le \cdots\le w_n\le\cdots \le \bw(t)<\infty
$$
for a.a.~$(x,t)\in \RN\times [0,T)$.
For each $t>0$, the limits $\lim_{k\to\infty}w_k(t)$ exists for a.a.~$(x,t)\in\RN\times[0,T)$.
Thus,
$$
w(t):=\lim_{k\to\infty}w_k(t)
$$
for a.a.~$(x,t)\in\RN\times [0,T)$.
We apply the monotone convergence theorem to \eqref{S5L1E1}.
Since $w(t)\le \bw(t)<\infty$ for a.a.~$(x,t)\in\RN\times[0,T)$, we obtain
$$
\infty>w(t)=S(t)w_0+\int_0^tS(t-s)f(w(s))\di s
$$
for a.a.~$(x,t)\in\RN\times[0,T)$, and hence $w$ is the solution of problem~\eqref{S1E1} in $\mathbb{R}^N\times[0,T)$ in the sense of Definition~\ref{S1D1}.
\end{proof}

\begin{proof}[Proof of Theorem~{\rm \ref{S1T1}}]
We show that problem~\eqref{S1E1} with $u_0=u^*$ has two positive solutions.
We have already shown in Lemma~\ref{S10L1} that $u^*$ is a solution of problem~\eqref{S1E1} with $u_0=u^*$ in the sense of Definition~\ref{S1D1}.
Hereafter, we show that the problem has another solution $u(t)$ such that $u(t)\in L^{\infty}_{loc}((0,t_0),L^{\infty}(\RN))$.

Let $u_{0n}(x):=\min\{ u^*(x),n\}$.
Then, $u_{0n}(x)\le v(x,0)$ a.e.~in $\RN$.
Let $\zeta_n(t)$ be a solution of $\zeta_n^{\prime}(t) = f(\zeta_n(t))$ for $t \in (0,t_n)$ and $\zeta_n(0) =n$, where $t_n$ is the maximal existence time of $\zeta_n$.
Then $u_{0n}(x)\le \zeta_n(0)$ in $\RN$.
Taking sufficiently large $n$ if necessary, we can assume that $t_n < t_0$, where $t_0$ is as in Lemma~\ref{S4L2}.
Since $v(t)$ and $\zeta_n(t)$ are  supersolutions in the sense of Definition~\ref{S1D1} (see Lemma~\ref{S4L2}), 
we have
\begin{equation*}
\begin{split}
\infty > v(t) 
&\ge S(t)v(0) + \int_0^t S(t-s) f(v(s)) \, \di s\\
&\ge S(t)u_{0n} + \int_0^t S(t-s) f(\min\{v(s), \zeta_n(s)\}) \, \di s
\end{split}
\end{equation*}
for a.a.~$(x,t) \in \mathbb{R}^N\times[0,t_0)$
and
\begin{equation*}
\begin{split}
\infty > \zeta_n(t) 
&= n + \int_0^t f(\zeta_n(s)) \, \di s\\
&= S(t)\zeta_n(0) + \int_0^t S(t-s) f(\zeta_n(s)) \, \di s\\
&\ge S(t)u_{0n} + \int_0^t S(t-s) f(\min\{v(s), \zeta_n(s)\}) \, \di s
\end{split}
\end{equation*}
for a.a.~$(x,t) \in \mathbb{R}^N\times[0,t_n)$.
Noting that $\min\{v(x,t), \zeta_n(t)\} = v(x,t)$ for a.a.~$\mathbb{R}^N\times[t_n,t_0)$,
we see that
these imply that
$\min\{v(x,t), \zeta_n(t)\}$ is a supersolution of problem~\eqref{S1E1} in $\mathbb{R}^N\times [0,t_0)$
in the sense of Definition~\ref{S1D1}.
Thus, by Lemma~\ref{S5L1} we see that problem~\eqref{S1E1} with initial data $u_{0n}$ has a solution $u_n(t)$ in $\mathbb{R}^N\times[0,t_0)$ such that
$$
u_n(t)\le v(t)\quad\text{a.e.~in}\ \RN\times[0,t_0)
$$
and
$$
u_n(t)\le \zeta_n(t)\quad\text{a.e.~in}\ \RN\times[0,t_n).
$$
Since $v\in L^\infty(\mathbb{R}^N\times (\delta,t_0))$ for any $\delta\in (0,t_0)$ and $\zeta_n(t) < \infty$ for all $t\in[0,t_n/2)$, we see that $u_n\in L^\infty(\mathbb{R}^N\times [0,t_0))$,
hence $u_n(t)$ is a classical solution of problem~\eqref{S1E1} in $\mathbb{R}^N\times[0,t_0)$.
We shall show that $u_n \in C(\mathbb{R}^N\times [0,t_0))$.
The solution $u_n$ satisfies
\begin{equation*}
u_n(t)=S(t)u_{0n}+\int_0^tS(t-s)f(u_n(s))\di s
\end{equation*}
for all $(x,t) \in \RN\times[0,t_0)$.
Noting that $u_n(t)\in C^{2,1}(\RN\times (0,t_0))\cap L^\infty(\mathbb{R}^N\times (0,t_0)) $, we have
\begin{equation*}
\|u_n(t)- u_{0n}\|_{C(\mathbb{R}^N)}
\le \|S(t)u_{0n} - u_{0n}\|_{C(\mathbb{R}^N)}+\int_0^t\|S(t-s)f(u_n)\|_{C(\mathbb{R}^N)} \, \di s
\end{equation*}
for all $t \in (0,t_0)$.
Since $u_{0n} \in L^\infty(\mathbb{R}^N) \cap BUC(\mathbb{R}^N)$,
the first term on the RHS tends to $0$ as $t\to 0^+$ (see {\it e.g.} \cite[Section~2.3]{E10}).
It suffices to consider the second term on the RHS, but by the smoothing effect of the heat kernel, it is obvious that
\[
\int_0^t\|S(t-s)f(u_n)\|_{C(\mathbb{R}^N)} \, \di s \le C t f(\|u_n\|_{L^\infty(\mathbb{R}^N \times (0,t_0))})
\to 0 \quad \mbox{as} \quad t\to 0^+.
\]
Thus, we see that $u_n \in C(\mathbb{R}^N\times[0,t_0))$.
Next, we shall show that $|\nabla u_n|$ is bounded on $\mathbb{R}^N \times (0,t_0)$.
Since $u$ is smooth, we have 
\begin{equation*}
\nabla u_n(t)= \nabla S(t)u_{0n}+\int_0^t \nabla S(t-s)f(u_n(s)) \, \di s
\end{equation*}
for all $(x,t) \in \RN\times(0,t_0)$.
Now by integration by parts,  we have
\begin{equation*}
\begin{split}
&\nabla S(t)u_{0n}(x) \\
&  = \nabla_x \left( \int_{\{u^*(y) \ge n\}} G_t(x-y) n \, \di y + \int_{\{u^*(y) < n\}} G_t(x-y) u^*(y) \, \di y\right)\\
&= - \int_{\partial \{u^*(y) \ge n\}} G_t(x-y)n \nu_1(y) \, \di {S_y}\\
&\qquad+ \int_{\{u^*(y) < n\}}  G_t(x-y) \nabla u^*(y) \, \di y
-  \int_{\partial \{u^*(y) < n\}} G_t(x-y)u^*(y) \nu_2(y) \, \di {S_y}\\
&= \int_{\{u^*(y) < n\}}  G_t(x-y) \nabla u^*(y) \, \di y
\end{split}
\end{equation*}
for all $(x,t) \in \mathbb{R}^N\times (0,t_0)$, where 
$\nu_1$ (resp.~$\nu_2$) is the outer normal unit vector on $\partial\{u^*(y) \ge n\}$ (resp. $\partial\{u^*(y) < n\}$). Note that 
\[
\partial\{u^*(y) \ge n\} = \partial \{u^*(y) < n\} = \{u^*(y) = n\}
\]
and
$\nu_1 = - \nu_2$ since $u^*$ is radial and decreasing.
Therefore, we have
\[
\|\nabla S(t) u_{0n}\|_{L^\infty(\mathbb{R}^N)} \le C\|\nabla u^*\|_{L^\infty (\{u^*(x)<n\})}<\infty
\]
for all $t \in (0,t_0)$.
This together with the smoothing effect of the heat kernel
implies  that
\begin{equation*}
\begin{split}
\|\nabla u_n(t)\|_{L^\infty(\mathbb{R}^N)}
&\le  C\|\nabla u^*\|_{L^\infty (\{u^*(x)<n\})} + C f(\|u_n\|_{L^\infty(\mathbb{R}^N \times (0,t_0))}) \int_0^t (t-s)^{-\frac{1}
{2}} \, \di s\\
&\le C\|\nabla u^*\|_{L^\infty (\{u^*(x)<n\})}+ C t^\frac{1}{2}f(\|u_n\|_{L^\infty(\mathbb{R}^N \times (0,t_0))})
\end{split}
\end{equation*}
for all $t\in (0,t_0)$, which implies that $\nabla u_n \in L^\infty (\mathbb{R}^N \times(0,t_0))$.

In summary, $u_n(t)\in C^{2,1}(\RN\times (0,t_0))\cap C(\RN\times[0,t_0))$, and $|u_n|$ and $|\nabla u_n|$ are bounded on $\RN\times(0,t_0)$.
Then, all assumptions of the comparison principle for classical solutions on $\RN$ (\cite[Proposition 52.6 in p.617]{QS19}) are satisfied.
We see that for $n\ge 2$,
$$
u_{n-1}(x,t)\le u_n(x,t)\quad\text{a.e.~in}~\RN\times [0,t_0).
$$
Therefore, there exists a limit
$$
u(x,t):=\lim_{n\to\infty} u_n(x,t)\quad\text{a.e.~in}\ \RN\times[0,t_0).
$$
The solution $u_n$ satisfies
\begin{equation}\label{S1T1PE2}
u_n(t)=S(t)u_{0n}(t)+\int_0^tS(t-s)f(u_n(x))\di s
\end{equation}
a.e.~in $\RN\times[0,t_0)$.
Letting $n\to\infty$ in \eqref{S1T1PE2}, by the monotone convergence theorem we have
\begin{equation}\label{S1T1PE3}  
u(t)=S(t)u^*+\int_0^tS(t-s)f(u(s))\di s
\end{equation}
a.e.~in $\RN\times[0,t_0)$.
Thus, $u(t)$ is a solution in the sense of Definition~\ref{S1D1}.
Since $u_n(t)\le v(t)$, we have
$$
u(x,t)\le v(x,t)\le\bu(0,t)=F^{-1}[c_0t]
$$
a.e.~in $\RN\times[0,t_0)$, where $c_0:=F_q[\varphi(0,\alpha_0)]$.
Then, $u(x,t)\le \min\{u^*(x),F^{-1}[c_0t]\}$ a.e.~in $\RN\times[0,t_0)$, and hence $u(t)\in L^{\infty}_{loc}((0,t_0),L^{\infty}(\RN))$.
By the parabolic regularity theorem, we see that $u$ is a classical solution of problem~\eqref{S1E1} with $u_0=u^*$ in $\mathbb{R}^N\times(0,t_0)$.

Finally, we shall prove \eqref{S1T1S1}.
Since $u$ is a classical solution, $u$ satisfies \eqref{S1T1PE3} for all $(x,t) \in \mathbb{R}^N\times(0,t_0)$.
For $1\le \gamma <\gamma^*$, we see from $u(x,t) \le u^*(x)$ for a.e.~$(x,t) \in \mathbb{R}^N\times(0,t_0)$ that
\begin{equation*}
\|u(t) - u^*\|_{L^\gamma_{ul}(\mathbb{R}^N)} \le 
\|S(t)u^* - u^*\|_{L^\gamma_{ul}(\mathbb{R}^N)}
+ \int_0^t \|S(t-s)f(u^*)\|_{L^\gamma_{ul}(\mathbb{R}^N)} \, \di s
\end{equation*}
for all $t\in(0,t_0)$.
It follows from Propositions~\ref{S1P1} and \ref{S4P2} that
the first term on the RHS tends to $0$ as $t\to 0^+$.
It suffices to consider the second term on the RHS.
Note that $q_f<q_0$ and $1\le \gamma <\gamma^*$ imply that
\[
\max\left\{1, \frac{N\gamma}{N+2\gamma}\right\} < \frac{N}{2q_f}.
\]
First, let us consider the case of $1<\gamma<\gamma^*$.
We can take a sufficiently small $\delta>0$ and 
$1< \beta < \gamma$ such that
\[
\max\left\{1, \frac{N\gamma}{N+2\gamma}\right\}< \beta  < \frac{N}{2q_f+2\delta}< \frac{N}{2q_f}.
\]
By \eqref{S3L3E2}, we have
\[
\beta < \frac{N}{2q_f+2\delta} \implies f(u^*) \in L^\beta_{ul}(\mathbb{R}^N).
\]
Furthermore, 
\[
\frac{N\gamma}{N+2\gamma} < \beta \implies 1-\frac{N}{2}\left(\frac{1}{\beta}-\frac{1}{\gamma}\right)>0.
\]
Thus, we see that
\begin{equation*}
\begin{split}
\int_0^t \|S(t-s)f(u^*)\|_{L^\gamma_{ul}(\mathbb{R}^N)} \, \di s
&\le C\|f(u^*)\|_{L^\beta_{ul}(\mathbb{R}^N)} \int_0^t \left((t-s)^{-\frac{N}{2}\left(\frac{1}{\beta}-\frac{1}{\gamma}\right)} +1 \right) \, \di s\\
& \to 0 \quad \mbox{as} \quad t \to 0^+,
\end{split}
\end{equation*}
and obtain \eqref{S1T1S1}.
Next, let us consider the case of $\gamma=1 (=\beta)$.
This  case is much simpler.
We take a sufficiently small $\delta>0$ such that
$1 < N/(2q_f+2\delta)$.
Thus, we see that
\begin{equation*}
\begin{split}
\int_0^t \|S(t-s)f(u^*)\|_{L^1_{ul}(\mathbb{R}^N)} \, \di s
&\le Ct\|f(u^*)\|_{L^1_{ul}(\mathbb{R}^N)} \to 0 \quad \mbox{as} \quad t \to 0^+,
\end{split}
\end{equation*}
and obtain \eqref{S1T1S1}.
The proof is complete.
\end{proof}

\section{Nonuniqueness (critical/subcritical)}
In order to study the profile of self-similar solutions we consider the problem
\begin{equation}\label{S6E1}
\begin{cases}
\displaystyle{\varphi''+\frac{N-1}{\eta}\varphi'+\frac{1}{2}\eta\varphi'+\frac{\varphi}{p-1}+\varphi^p=0}, & \text{for}\ \eta>0,\\
\varphi(0)=\alpha>0,\ \varphi'(0)=0.
\end{cases}
\end{equation}
It was proved in \cite{HW82} that $\varphi(\eta,\alpha)>0$ for $\eta\ge 0$ if $p>p_F$.

\begin{proposition}[{\cite[Corollary~1.2]{N06}}]
Let $p_F<p<p_{JL}$.
Then the following holds:
\begin{equation}\label{S6P1E1}
\sup\{\ell>0;\ S_{\ell}\neq\emptyset\}>L,
\end{equation}
where $L$ is defined by \eqref{L} and
$$
S_{\ell} = \left\{\varphi;\ \textrm{$\varphi(\eta,\alpha)$ is a positive solution of \eqref{S6E1} and }
\lim_{\eta\to\infty}\eta^{\frac{2}{p-1}}\varphi(\eta,\alpha)=\ell
\right\}.
$$
\end{proposition}
In Section~4 we have already seen that
$$
\varphi^*(\eta):=F_q^{-1}\left[\frac{\eta^2}{2N-4q_f}\right]=L\eta^{-\frac{2}{p-1}}
$$
is a singular solution of equation in~\eqref{S6E1}.
It follows from \eqref{S6P1E1} that there exists $\alpha_0>0$ such that $\varphi(\eta,\alpha_0)$ and $\varphi^*(\eta)$ has the first intersection point in $\eta>0$.
Thus, Proposition~\ref{S4P1} is valid for $(p_F<)p_0<p<p_{JL}$.

\begin{proof}[Proof of Theorem{\rm ~\ref{S1T2}}]
Because of (A8), we assume the existence of a singular stationary solution $u^*$.
Since $q_f<q_0$ and $N>2$, the proof of Lemma~\ref{S3L3} is valid, and hence \eqref{S3L3E0} holds even in the case $q_S\le q_f<q_0$.

Since $q_f<q_0$, two inequalities
$$
N-2q_f-2\delta>0\quad\text{and}\quad N+1-2q_f-2\delta>0
$$
hold for small $\delta>0$.
Therefore, \eqref{S10L1E2} holds, and the other arguments in the proof of Lemma~\ref{S10L1} are valid.
Thus, Lemma~\ref{S10L1} holds.
Specifically, $u^*$ is a positive radial singular stationary solution of equation in~\eqref{S1E1} in the sense of Definition~\ref{S1D1} even in the case $q_S\le q_f<q_0$.

It follows from Proposition~\ref{S4P1}, which is valid for $p_0<p<p_{JL}$, holds for $q_S\le q<q_0$.
Then, all the arguments in Section~4 are valid in the case $q_S\le q_f<q_0$.
In particular, $v$ defined by \eqref{S4E7} is a supersolution of problem~\eqref{S1E1} in $\RN\times[0,t_0)$ in the sense of Definition~\ref{S1D1}.

Using $v$, we can construct a bounded solution $u(t)$ of problem~\eqref{S1E1} with $u_0=u^*$.
The proof of the construction is the same as the proof of Theorem~\ref{S1T1}.
In particular, $u(t)\le u^*$ a.e.~in $\RN$ for $0<t<t_0$.
Using this inequality, in the same way as Theorem~\ref{S1T1} we can obtain \eqref{S1T1S1} provided that $1\le \gamma<\gamma^*$.

Now, we have two solutions, namely $u^*$ and $u(t)$.
The proof is complete.
\end{proof}

\bigskip
\noindent
{\bf Acknowledgement}\\
YM was supported by JSPS KAKENHI Grant Number 24K00530.\\

\noindent{\bf Conflict of interest}\\
The authors have no relevant financial or non-financial interests to disclose.\\

\noindent{\bf Data Availability}\\
Data sharing is not applicable to this article as no new data were created or analyzed in this study.



\end{document}